\documentclass[twoside]{article}

\usepackage[preprint]{aistats2026}

\usepackage[utf8]{inputenc}
\usepackage[T1]{fontenc}
\usepackage[round]{natbib}

\usepackage{amsfonts}
\usepackage{amssymb}
\usepackage{mathtools}
\usepackage{amsthm}
\usepackage{aliascnt}
\usepackage{booktabs}
\usepackage{nicefrac}
\usepackage{graphicx}
\usepackage{enumitem}
\usepackage{microtype}
\usepackage{xcolor}
\usepackage{xparse}
\usepackage{tikz}
\usepackage{algorithm}
\usepackage{algpseudocode}
\usepackage{float}

\newcommand\blfootnote[1]{%
  \begingroup
  \renewcommand\thefootnote{}\footnote{#1}%
  \addtocounter{footnote}{-1}%
  \endgroup
}

\let\epsilon\varepsilon

\usepackage{amsmath}

\usepackage{url}
\usepackage{xurl}
\usepackage[colorlinks=true,allcolors=blue,linktoc=all]{hyperref}
\usepackage[nameinlink,capitalize,noabbrev]{cleveref}

\theoremstyle{plain}
\newtheorem{theorem}{Theorem}[section]
\newaliascnt{proposition}{theorem}
\newtheorem{proposition}[proposition]{Proposition}
\aliascntresetthe{proposition}
\newaliascnt{lemma}{theorem}
\newtheorem{lemma}[lemma]{Lemma}
\aliascntresetthe{lemma}
\newaliascnt{corollary}{theorem}

\aliascntresetthe{corollary}
\newaliascnt{fact}{theorem}
\newtheorem{fact}[fact]{Fact}
\aliascntresetthe{fact}
\theoremstyle{definition}
\newaliascnt{definition}{theorem}
\newtheorem{definition}[definition]{Definition}
\aliascntresetthe{definition}
\newaliascnt{assumption}{theorem}
\newtheorem{assumption}[assumption]{Assumption}
\aliascntresetthe{assumption}
\theoremstyle{remark}
\newaliascnt{remark}{theorem}
\newtheorem{remark}[remark]{Remark}
\aliascntresetthe{remark}

\crefname{equation}{}{}
\crefname{theorem}{Theorem}{Theorems}
\Crefname{theorem}{Theorem}{Theorems}
\crefname{proposition}{Proposition}{Propositions}
\Crefname{proposition}{Proposition}{Propositions}
\crefname{lemma}{Lemma}{Lemmas}
\Crefname{lemma}{Lemma}{Lemmas}
\crefname{corollary}{Corollary}{Corollaries}
\Crefname{corollary}{Corollary}{Corollaries}
\crefname{assumption}{Assumption}{Assumptions}
\Crefname{assumption}{Assumption}{Assumptions}
\crefname{definition}{Definition}{Definitions}
\Crefname{definition}{Definition}{Definitions}
\crefname{fact}{Fact}{Facts}
\Crefname{fact}{Fact}{Facts}
\crefname{remark}{Remark}{Remarks}
\Crefname{remark}{Remark}{Remarks}
\Crefname{ALG@line}{Line}{Lines}
\crefname{ALG@line}{Line}{Lines}

\makeatletter
\providecommand{\theHALG@line}{}
\renewcommand{\theHALG@line}{\thealgorithm.\arabic{ALG@line}}
\makeatother

\DeclareMathOperator*{\argmin}{arg\,min}

\newcommand{\R}{\mathbb{R}}
\newcommand{\innp}[1]{\langle #1 \rangle}
\newcommand{\norm}[1]{\| #1 \|}

\newcommand{\grad}{\nabla}

\DeclareRobustCommand{\defi}{\stackrel{\mathrm{\scriptscriptstyle def}}{=}}

\newcommand{\poly}{\operatorname{poly}}

\newcommand{\AGD}{\textnormal{AGD}}
\newcommand{\AGDp}{ \hyperref[alg:generalized-agd]{\textnormal{AGD+}}}

\makeatletter
\newcommand{\newtarget}[2]{\Hy@raisedlink{\hypertarget{#1}{}}#2}
\makeatother
\newcommand{\linktoproof}[1]{{\normalfont[{\hyperlink{proof:#1}{$\downarrow$}}]}}
\newcommand{\linkofproof}[1]{%
    \newtarget{proof:#1}\textbf{Proof of \cref{#1}}%
}

\newcommand*\circledaux[1]{%
    \tikz[baseline=(char.base)]{%
        \node[shape=circle,draw,inner sep=0.8pt] (char) {#1};%
    }%
}
\NewDocumentCommand{\circled}{m o}{%
    \IfNoValueTF{#2}{\circledaux{#1}}{\stackrel{\circledaux{#1}}{#2}}%
}

\usepackage{hyperref}
\usepackage{booktabs}
\usepackage{tabularx}
\usepackage{array}
\usepackage{makecell}
\usepackage{threeparttable}
\usepackage{subcaption}

\newcolumntype{L}[1]{%
    >{\raggedright\arraybackslash}m{#1}%
}
\newcolumntype{C}[1]{%
    >{\centering\arraybackslash}m{#1}%
}
\newcommand{\pnorm}[1]{\left\lVert #1\right\rVert_p}
\newcommand{\dnorm}[1]{\left\lVert #1\right\rVert_{p^*}}
\newcommand{\eps}{\varepsilon}
\newcommand{\Prox}{\mathcal{P}}

\begin{document}

\runningtitle{Optimal Non-Euclidean Gradient-Norm Minimization}
\runningauthor{Pelleriti et al.}

\twocolumn[

\aistatstitle{Optimal Gradient-Norm Minimization in Non-Euclidean H\"older-Smooth Convex Optimization}

\aistatsauthor{
Nico Pelleriti\textsuperscript{1,2}
\
Maryam Shiran
\
David Mart\'inez-Rubio\textsuperscript{3}
\
Max Zimmer\textsuperscript{1,2}
\
Sebastian Pokutta\textsuperscript{1,2}
}

\aistatsaddress{
\textsuperscript{1}Zuse Institute Berlin, Germany
\quad
\textsuperscript{2}Technical University of Berlin, Germany
\\
\textsuperscript{3}IMDEA Software Institute, Madrid, Spain\\
}
]

\blfootnote{A preliminary version of this work was submitted to ICML 2026 and NeurIPS 2026, obtaining optimal rates for the smooth problem and improved rates for the H\"older-smooth case. The current version contains a substantial improvement by obtaining near-optimal complexity for the latter problem.}

\begin{abstract}
Minimizing gradients of a convex function is an important problem across optimization and learning tasks. The gradient provides a directly computable certificate of approximate stationarity, and its minimization usually implies stronger results than those for minimization of function values.
In this work, we study gradient-norm minimization for convex functions that are
$(L,\kappa)$-H\"older smooth with respect to the $\ell_p$-norms,
$p \geq 1$. We develop algorithms that achieve near-optimal gradient-oracle complexity for this problem. In the smooth case, our results resolve the previously open setting $p>2$. For H\"older-smooth objectives,
we close the complexity gap throughout the full $p$-range, including to the best of our knowledge,
a gap in the Euclidean case.
We provide two families of algorithms: the first one comes with a simple iteration and generalizes a phenomenon known as mirror duality, exploiting dual behaviours of algorithms with errors and inexact computations. The second makes use of accumulating regularizers centered at different approximate solutions, which we sequentially minimize in order to provide our near-optimal rates.
\end{abstract}

\section{Introduction}

Minimizing a lower bounded convex function $f$ is a central task in continuous optimization and underlies numerous learning problems.
A typical performance measure, when a minimizer $x^\ast$ exists, is the function value gap $f(x)-f(x^\ast)$. However, this quantity is typically \emph{not directly checkable} because it depends on the unknown optimum value $f(x^\ast)$.
In contrast, the gradient norm $\|\nabla f(x)\|_*$ is directly computable and provides an explicit certificate of approximate stationarity.
This property makes \emph{gradient-norm minimization} a natural objective when one needs a verifiable stopping criterion.

\begin{table*}[t]
\centering
\scriptsize
\setlength{\tabcolsep}{2pt}
\renewcommand{\arraystretch}{1.65}

\begin{tabularx}{\linewidth}{
@{}
>{\raggedright\arraybackslash}p{0.17\linewidth}
>{\centering\arraybackslash}X
>{\centering\arraybackslash}X
>{\centering\arraybackslash}X
@{}
}
\toprule
\textbf{Setting}
&
\textbf{Lower Bound}
&
\textbf{Our Upper Bounds}
&
\textbf{Previous Upper Bound}
\\
\midrule

$(L,\kappa)$-H\"older-smooth,\par
$1\leq p\leq2$,\par
$1<\kappa\leq2$
&
$\displaystyle
\widetilde{\Omega}_{p,\kappa}\!\left(
\left(
\frac{LR^{\kappa-1}}{\epsilon}
\right)^{
\frac{2}{3\kappa-2}
}
\right)$
\par\smallskip
{\scriptsize \hyperlink{cite.diakonikolas2024complementary}{DG24}, Corollary~2}
&
$\displaystyle
\widetilde{O}_{p,\kappa}\!\left(
\left(
\frac{
LR^{\kappa-1}
}{
\epsilon
}
\right)^{
\frac{2}{3\kappa-2}
}
\right)$
\par\smallskip
{\scriptsize
    \cref{thm:accumulative-power}}
&
$\displaystyle
\widetilde O_{p,\kappa}\!\left(
\left(
\frac{
LR^{\kappa-1}
}{
\epsilon
}
\right)^{
\frac{\kappa}
{(\kappa-1)(3\kappa-2)}
}
\right)$
\par\smallskip
{\scriptsize (\hyperlink{cite.diakonikolas2024complementary}{DG24}, Thm.~3), (\hyperlink{cite.kim2023mirror}{KJC+23}, Cor.~3)}
\\

\midrule

$(L,\kappa)$-H\"older-smooth,\par
$2<p<\infty$,\par
$1<\kappa\leq2$
&
$\displaystyle
\widetilde{\Omega}_{p,\kappa}\!\left(
\left(
\frac{LR^{\kappa-1}}{\epsilon}
\right)^{
\frac{p}
{(p+1)\kappa- p}
}
\right)$
\par\smallskip
{\scriptsize \hyperlink{cite.diakonikolas2024complementary}{DG24}, Corollary~2}
&
$\displaystyle
\widetilde{O}_{p,\kappa}\!\left(
\left(
\frac{
LR^{\kappa-1}
}{
\epsilon
}
\right)^{
\frac{p}
{( p+1)\kappa- p}
}
\right)$
\par\smallskip
{\scriptsize
\cref{thm:mirror_final_rate}$^\ast$, \cref{thm:accumulative-power}}
&
$\displaystyle
\widetilde O_{p,\kappa}\!\left(
\left(
\frac{
LR^{\kappa-1}
}{
\epsilon
}
\right)^{
\frac{
\kappa(p-1)
}{
(\kappa-1)((p+1)\kappa-p)
}
}
\right)$
\par\smallskip
{\scriptsize \hyperlink{cite.diakonikolas2024complementary}{DG24}, Theorem~3}
\\

\bottomrule
\end{tabularx}

\vspace{0.4em}

\caption{
Gradient complexity for $\norm{\cdot}_{p^\ast}$-gradient-minimization of $(L,\kappa)$-H\"older-smooth
convex functions initialized at a distance $R$ from a minimizer. The smooth case corresponds to $\kappa=2$. For \cref{thm:mirror_final_rate}$^\ast$, our rate above only applies to this case, while \cref{thm:accumulative-power} applies generally.
Our upper bounds only contain log factors for $p=1$.
We wrote the upper bounds in \citep[Theorem~3]{diakonikolas2024complementary} with a minor correction to their reported rate. The rate in \citep[Corollary~3]{kim2023mirror} only contains log factors for $p=1$.}
\label{tab:main_results}

\end{table*}

Gradient-based certificates are also the right criteria to target for some problems beyond minimization.
For example, saddle-point problems and variational inequalities are often assessed via stationarity or residual measures rather than primal function suboptimality \citep{abernethy2021lastiterate,diakonikolas2020halpern,yoon2021accelerated,gorbunov2022extragradient}.
In some cases of unconstrained minimization, the gradient norm and function value criteria are related, both capturing approximate optimality in a quantitative way. An important case occurs when $f$ is $L$-smooth with respect to a norm $\|\cdot\|$. In that case, for all $x$,
$
\|\nabla f(x)\|_*^2 \le 2L\bigl(f(x)-f(x^\ast)\bigr),
$
so a small function gap implies a small gradient norm.
Conversely, if $f$ has additional curvature (e.g., $\mu$-strong convexity), then a small gradient norm implies near-optimal function value:
$
f(x)-f(x^\ast) \le \frac{1}{2\mu}\|\nabla f(x)\|_*^2.
$

Despite this close relationship, optimal methods for convex smooth minimization do not automatically yield optimal guarantees for gradient minimization.
Even in the classical Euclidean smooth setting, specialized arguments are used to obtain a sharp complexity for finding $\|\nabla f(x)\|_2 \le \varepsilon$ \citep{nesterov2012make}.
The situation becomes more delicate in \emph{non-Euclidean} geometry: when smoothness is measured in a $p$-norm, the natural certificate becomes the gradient's dual norm $\|\nabla f(x)\|_{p^\ast}$. Prior techniques achieve optimal algorithms for this task in the case $p\in[1,2]$ but leave a gap for $p>2$ highlighted by \citet{diakonikolas2024complementary}. Besides, under the more general Hölder smooth assumption, the latter paper highlights a gap between their lower bounds and known upper bounds in all regimes, even for Euclidean norms. The gap in the Euclidean case was also highlighted by \citep{dvurechensky2024nearoptimal}.

Another motivation to look into gradient-minimization algorithms is that they appear to be stronger than those for function minimization in some cases, since the simple reduction using convexity and H\"older's inequality
$
    f(x)-f(x^\ast) \leq \innp{\nabla f(x), x-x^\ast} \leq \norm{\nabla f(x)}_\ast\norm{x-x^\ast},
$
allows for achieving optimal rates for function minimization from optimal rates from gradient minimization in the cases where optimal rates have been achieved in both problem classes in bounded domains \citep{diakonikolas2024complementary}. This argument also allows one to transfer lower bounds developed for function minimization to gradient minimization.

Beyond stopping criteria, gradient-norm minimization in non-Euclidean norms underlies several machine-learning applications catalogued by \citet{diakonikolas2024complementary}: controlling the maximum residual in linear systems, sparsity in $\ell_p$-regression, and risk minimization with near-minimal predictor norm. The entropy-regularized Optimal Transport dual is $\log d$-smooth with respect to $\norm{\cdot}_1$ which falls into the  regime that we treat in this paper.

In this work, we close the gap in the smooth case for $p>2$ and obtain near-optimal gradient-norm
complexity for H\"older-smooth objectives throughout
$p\in[1,\infty)$. We develop two complementary approaches.
Our first approach builds on mirror duality \citep{kim2023timereversed}, which turns an accelerated primal method equipped with a uniformly convex regularizer $\psi$ into a time-reversed dual method whose objective controls a gradient certificate. We first use the primal method as a warm start to reduce function suboptimality, and then initialize the mirror-dual method from its output to reduce the gradient norm. The resulting dual iterate controls $\psi^\ast(\nabla f(\cdot))$, which in turn bounds $\|\nabla f(\cdot)\|_{p^\ast}$. To handle non-strongly convex regularizers in the $p>2$ geometry, we obtain a generalization of mirror duality using inexactly smooth or inexactly strongly convex functions \citep{devolder2014first} which we then can apply to the gradient minimization problem, enabling accelerated analyses by controlling the resulting inexactness terms. This algorithm has a cheap per-iteration complexity but it only matches known lower bounds for the smooth case, whereas we only obtained suboptimal rates for the general Hölder-smooth case.

Our second approach designs an algorithm by adding accumulative regularization to the function, that is increasing over time. We solve a sequence these increasingly regularized objectives, adding shifted power regularizers centered at approximate solutions from earlier stages. Uniform convexity converts the function-value accuracy at each stage into bounds for their iterates, while the optimality condition of the final accumulated objective converts these bounds into a bound on the gradient of the original objective. Combining this schedule with a generalized accelerated method for the regularized subproblems yields the optimal H\"older-smooth exponent in general $\ell_p$ geometry.

\paragraph{Contributions.} Our contributions can be summarized as follows:
\begin{enumerate}[leftmargin=*,itemsep=2pt]
    \item \textbf{Optimal gradient-norm rates for $p>2$ in the smooth case.}
    We give accelerated first-order methods for minimizing $\|\nabla f(x)\|_{p^\ast}$ when $f$ is convex and $L$-smooth with respect to $\|\cdot\|_p$ for general $p>2$.
        Our complexity matches the known lower bound up to logarithmic factors, resolving the open regime identified by \citet{diakonikolas2024complementary}. Both of our approaches achieve this rate although the one in \cref{thm:mirror_final_rate} has a simpler per-iteration subproblem.

    \item \textbf{H\"older-smooth objectives via inexact smoothness.}
        Our approach in \cref{thm:accumulative-power} achieves the optimal H\"older-smooth complexity exponent throughout $p \in [1, \infty)$, solving the remaining open questions in  \citet{diakonikolas2024complementary}, where a gap existed even in the Euclidean case. Our guarantees interpolate smoothly between the smooth and less-smooth settings.
\end{enumerate}

\section{Related Work}

\citet{nesterov2012make} argued that optimizing a strongly-convex regularized objective with accelerated function-minimization methods results in a near-optimal rate for the problem of reducing $\norm{\nabla f(x)}_2$ for convex smooth problems with respect to the Euclidean norm. \citet{diakonikolas2024complementary} developed an accelerated algorithm for composite problems $f+\psi$, where $f$ is smooth and the sum is strongly convex with respect to non-Euclidean norms. Using their framework, they generalize the technique in \citet{nesterov2012make} for $p$-norms, where $p\in [1,\infty]$, obtaining near-optimal algorithms for the case $p\in[1,2]$, leaving the case $p > 2$ as an open question, which we resolve.

Later, \citet{kim2023convergence,kim2023timereversed} developed continuous-time frameworks for analyzing models of first-order methods in the form of ordinary differential equations, based on the performance estimation problems (PEP) of \citet{drori2014performance}, discovering that a specific rate for a function-minimization algorithm in their setting would translate to the same rate for gradient-norm minimization for a specific dual method. Both were developed for the Euclidean setting. \citet{kim2023timereversed} also provide equivalences between algorithms in discrete times, and a follow-up work \citep{kim2023mirror} generalizes the framework to non-Euclidean norms and applies it to the minimization of $\norm{\nabla f(\cdot)}_{p^\ast}$ for convex smooth functions with respect to $p$-norms, $p \in [1, 2]$. The result is an algorithm that is optimal up to logarithmic factors, improving over the prior result on this problem by \citet{diakonikolas2024complementary}. Their application to convex objectives smooth with respect to $\ell_p$ norms gives optimal gradient-norm complexity up to logarithmic factors for $p\in[1,2]$, improving over the earlier regularization bound in that regime. Our mirror-dual technique generalizes this phenomenon to obtain dimension-independent results for the $p>2$ setting.

\citet{devolder2014first} developed methods for optimization with an inexact oracle and showed how H\"older-smooth objectives can be treated through an inexact smoothness model. We take inspiration from this viewpoint on both sides of the mirror-duality construction: H\"older-smooth objectives are represented as inexactly smooth functions, while uniformly convex regularizers are treated through the corresponding inexact strong-convexity relation. We have to make a careful tuning of the inexact parameters in our algorithms to achieve optimality.

Another family of regularization algorithms solve a sequence of increasingly regularized problems rather than choosing a single regularization parameter. In stochastic Euclidean optimization, \citet{allenzhu2018make} introduced a recursive quadratic-regularization scheme whose successive regularizers are centered at approximate solutions of the preceding subproblems. \citet{foster2019complexity} extended this approach and established nearly matching upper and lower oracle-complexity bounds for stochastic convex optimization. In deterministic smooth optimization, \citet{lan2023optimal} developed accumulative-regularization methods attaining optimal first-order complexity for gradient minimization, together with parameter-free and more general variants. More recently, \citet{ji2025high} extended the framework to high-order methods, showing how fast function-value convergence of an inner method can be converted into gradient- or subgradient-norm guarantees. They require the order of the method plus the exponent of the Hölder-smoothness condition to be $\geq 2$, thus not including the Hölder-smooth case for first-order methods.

\section{Preliminaries and Groundwork}\label{sec:preliminaries}

We denote by $\|\cdot\|$ a norm on $\R^d$ and by $\|\cdot\|_*$ its
dual, where $\|z\|_* \defi \sup_{\|x\|\leq 1}\innp{z,x}$. We use $\log_+(\cdot) = \max\{0, \log_2(\cdot)\}$. For
differentiable and convex $f$, its Bregman divergence is
$D_f(x,y)\defi f(x)-f(y)-\innp{\nabla f(y),x-y} \geq 0$, and its Fenchel
conjugate is $f^\ast(u)\defi\sup_x\{\innp{u,x}-f(x)\}$.

We write $T$ for the total number of iterations. For $p\in[1,\infty]$,
we let $p^\ast\defi(1-1/p)^{-1}$ be the Young conjugate, with the
limiting conventions $p^\ast=\infty$ for $p=1$ and $p^\ast=1$ for
$p=\infty$. Then $\|\cdot\|_{p^\ast}$ is dual to $\|\cdot\|_p$.
When regularity is measured with respect to a $p$-norm,
$O_{p,\kappa}(\cdot)$ and $\Omega_{p,\kappa}(\cdot)$ may suppress constants
depending only on $p$ and the H\"older exponent $\kappa$. The notation
$\widetilde{O}_{p,\kappa}(\cdot)$ additionally suppresses logarithmic
factors in the problem parameters. 

\begin{definition}[Uniform convexity and H\"older smoothness]
\label{def:uniform-convexity-holder-smoothness}
Let \(f:\R^d\to\R\) be convex and differentiable. We say that \(f\) is
    \((\mu,q)\)-uniformly convex (resp. \((L,\kappa)\)-H\"older smooth) with respect to a norm \(\|\cdot\|\) if
    \circled{1} holds (resp. if \circled{2} holds)  for all $x,y\in\R^d$:
\[
    \frac{\mu}{q}\|x-y\|^q
    \circled{1}[\leq]
    D_f(x,y)
    \circled{2}[\leq]
    \frac{L}{\kappa}\|x-y\|^\kappa,
\]
where \(\mu>0\), \(q\geq2\), \(L>0\), and
\(\kappa\in(1,2]\).
When \(\kappa=2\), we simply say that \(f\) is \(L\)-smooth.
No nonconstant function satisfies the uniform-convexity inequality with
\(q<2\), cf. \citep{nesterov04introductory}.
\end{definition}

Uniform convexity generalizes strong convexity, which is the case $q=2$.
For $q > 2$, uniformly convex functions are flatter around the minimum than strongly convex functions. This property is crucial for analyzing algorithms in non-Euclidean spaces where strongly convex regularizers in bounded regions only exist if either the range of the function or $1/\mu$ depends polynomially on the dimension, which would lead to rates with $\poly(d)$ dependence. This is the case for $p$-norms, where $p > 2$ and $p$ is not trivially close to $2$, cf. \citet[Example 5.1]{daspremont2018optimal}.

\subsection{Regularizers and Accumulative Problems}
\label{subsec:powerregshifted}
Let
\(q\coloneqq\max\{2,p\}\) and
\(r_q(u)\coloneqq \frac{1}{q}\|u\|_p^q\), for $p \in (1, \infty)$. For $p \in \{1, \infty\}$, we will use a proxy norm, cf. \cref{rem:endpoint-norms}. Our algorithm in \cref{sec:accumulative_regularization} will make use of standard properties of \(r_q\), stated in the following fact.

\begin{fact}[Power-regularizer geometry]
\label{fact:power_geometry}\linktoproof{fact:power_geometry}
Let \(p\in(1,\infty)\) and \(q=\max\{2,p\}\). Then $r_q$ is
$(\mu_{p,q},q)$-uniformly convex with respect to $\|\cdot\|_p$, where
\[
    \mu_{p,q}\coloneqq
    \begin{cases}
        p-1, & 1<p\leq2,\\
        2^{2-p}, & p\geq2.
    \end{cases}
\]
Moreover, for every $u\in\mathbb{R}^d$,
\[
    \|\nabla r_q(u)\|_{p^\ast}=\|u\|_p^{q-1}.
\]
\end{fact}

Given \(z\in\mathbb{R}^d\), centers
\(c_1,\ldots,c_m\in\mathbb{R}^d\), and weights
\(\alpha_1,\ldots,\alpha_m\geq0\) with
\(\sum_{i=1}^m\alpha_i>0\), 
our algorithm in \cref{sec:accumulative_regularization} will compute iterations with the following structure
\begin{equation}
\label{eq:shifted_power_call}
    \operatorname*{arg\,min}_{y\in\mathbb{R}^d}
    \left\{
        \langle z,y\rangle
        +\sum_{i=1}^m\alpha_i r_q(y-c_i)
    \right\}.
\end{equation}
By \cref{fact:power_geometry}, the objective is coercive and strictly
convex, so the minimizer is well defined and unique. We call an exact
evaluation of \cref{eq:shifted_power_call} a \emph{shifted-power call}
and all of our uses of it will not require gradient evaluations of our objective function $f$. 

The oracle has a simple structure in several regimes. If
\(p=q=2\), letting
\(B\coloneqq\sum_{i=1}^m\alpha_i\) and
\(\bar c\coloneqq B^{-1}\sum_{i=1}^m\alpha_i c_i\), completing the
square gives that \cref{eq:shifted_power_call} corresponds to the gradient descent step:
\[
\bar c-\frac{z}{B}.
\]
If \(p>2\), then \(q=p\) and the problem separates across coordinates:
the \(j\)-th coordinate of the minimizer is the unique solution
\(t\in\mathbb{R}\) of
\begin{equation}
\label{eq:shifted-power-coordinate}
    z_j+\sum_{i=1}^m
    \alpha_i|t-c_{i,j}|^{p-2}(t-c_{i,j})=0.
\end{equation}
Indeed, the left-hand side is continuous and strictly increasing from
\(-\infty\) to \(+\infty\), so each shifted-power call reduces to \(d\)
independent one-dimensional root-finding problems.

For \(1<p<2\), \(q=2\) and
\(r_2(y-c)=\frac12\|y-c\|_p^2\) is generally not separable across
coordinates.

\subsection{Inexact Smoothness}

\begin{definition}[Inexact smoothness and strong convexity]
\label{def:inexact_holder_smooth}
For $L,\mu>0$ and $\delta\geq0$, a differentiable function $f$ is
$\delta$-inexact $\mu$-strongly convex (respectively,
    $\delta$-inexact $L$-smooth) with respect to $\|\cdot\|$ if $\circled{1}$ holds (resp. $\circled{2}$) for all
$x,y$,
\[
    \frac{\mu}{2}\|x-y\|^2-\delta
   \circled{1}[\leq]
    D_f(x,y)
    \circled{2}[\leq]
    \frac{L}{2}\|x-y\|^2+\delta.
\]
\end{definition}

H\"older-smooth functions with $\kappa\in(1,2)$ can be represented as
inexactly smooth functions. Similarly, uniformly convex functions can
be represented as inexactly strongly convex functions. These reductions,
the former introduced by \citet{devolder2014first}, allow us to adapt
proof techniques developed for the standard smooth and strongly convex
setting.

\begin{lemma}[Reduction to inexact smoothness]
\label{lemma:holder_inexact}\linktoproof{lemma:holder_inexact}
If $f$ is $(L,\kappa)$-H\"older smooth with $\kappa\in(1,2)$, then for
any $\gamma>0$, $f$ is $\delta$-inexact $\widetilde L$-smooth with
\[
    \widetilde L=\frac{L}{\gamma^{2/\kappa}},
    \qquad
    \delta
    =
    L\gamma^{2/(2-\kappa)}
    \frac{2-\kappa}{2\kappa}.
\]
\end{lemma}

Key to our mirror-dual analysis is the duality between uniform convexity
and H\"older smoothness. See, for example, \citet{zalinescu2002convex}. We provide a proof for completeness.

\begin{fact}[Uniform convexity and conjugate H\"older smoothness]
\label{prop:conjugate_holder}\linktoproof{prop:conjugate_holder}
The Fenchel conjugate $\psi^\ast$ of a differentiable
$(\mu,p)$-uniformly convex function $\psi$, with $p\geq2$, is
$(M,p^\ast)$-H\"older smooth with
\[
    M=\left(\frac{p}{2\mu}\right)^{p^\ast-1},
    \qquad
    p^\ast=\frac{p}{p-1}.
\]
\end{fact}

\section{Inexact Mirror Duality}\label{sec:mirror_duality}
Our first technique builds on mirror duality \citep{kim2023mirror}, which relates a primal first-order method to a time-reversed dual method whose objective controls a gradient certificate. We adapt and extend this correspondence to the inexact smoothness models arising from H\"older-smooth objectives and uniformly convex regularizers.

\begin{assumption}[Convex Hölder-smooth]
\label{ass:holder-setting}
Let \(p\in[1,\infty)\) and \(1<\kappa\le 2\). The function
\(f:\mathbb{R}^d\to\mathbb{R}\) is convex and differentiable, attains
its minimum at some \(x^\ast\in\mathbb{R}^d\), and is
\((L,\kappa)\)-Hölder smooth with respect to the \(\ell_p\)-norm.
\end{assumption}

Throughout this subsection, in addition to
\cref{ass:holder-setting}, the direct finite-dimensional construction
uses \(2\le p<\infty\) and the regularizer
\(
    \psi(u)\defi\frac{1}{p}\|u\|_p^p.
\)
Its uniform convexity makes $\psi^\ast$ H\"older smooth and converts a
bound on $\psi^\ast(z)$ into a bound on $\|z\|_{p^\ast}$. Mirror duality
\citep{kim2023mirror} time-reverses the
coefficient matrices of accelerated gradient descent and swaps its two
oracle maps, $\nabla f$ and $\nabla\psi^\ast$. Denote the resulting
method by $\mathsf{MD}(\mathsf{AGD})$. Its pseudocode is given in
\cref{alg:mirror-dual-agd}. Its final dual iterate has the important
algebraic property
\(
    r_T=\nabla f(q_T),
\)
so a dual function-value guarantee is a stationarity certificate at a
known primal point.

The standard mirror-duality proof uses quadratic smoothness remainders.
Here both $f$ and $\psi^\ast$ may only be H\"older smooth. The
inexactness levels and complete potential analysis are given in
\cref{lem:appendix_mirror_phase}. Applying
\cref{lemma:holder_inexact} to the two remainders yields, for suitable
weights $A_t\asymp t^2$, the compact potential bound:
\begin{equation}\label{eq:mirror_potential_summary}
    \psi^\ast(r_T)
    \lesssim
    \frac{\Delta}{A_T}
    +T\delta_\psi
    +\delta_f\sum_{t=1}^{T-1}\frac{1}{A_t}.
\end{equation}
Here $\Delta\defi f(q_0)-f(x^\ast)$, while $\delta_\psi$ and $\delta_f$
are the tunable inexactness levels for $\psi^\ast$ and $f$, respectively.
Since
$\sum_{t\geq1}A_t^{-1}<\infty$, optimizing these levels gives
\begin{equation}\label{eq:mirror_phase_rate}
    \|\nabla f(q_T)\|_{p^\ast}
    =
    O_{p,\kappa}\!\left(
        \frac{L^{1/\kappa}
        \Delta^{(\kappa-1)/\kappa}}
        {T^{\,3/2+1/p-2/\kappa}}
    \right),
    \qquad 1<\kappa<2,
\end{equation}
whereas for $\kappa=2$ the bound is
$O_p\bigl(\sqrt{L\Delta}/T^{1/2+1/p}\bigr)$.

The dependence on the initial function gap is removed by a short primal
warm start, formalized in \cref{lem:appendix_primal_warm_start} and
implemented in \cref{alg:primal-warm-start}. Namely, run non-Euclidean accelerated gradient descent for
$T$ iterations from $x_0$, then initialize
$\mathsf{MD}(\mathsf{AGD})$ at its output and run it for another $T$
iterations. The first phase ensures
\[
    \Delta
    =
    O_{p,\kappa}\!\left(
        \frac{LR^\kappa}
        {T^{((p+1)\kappa-p)/p}}
    \right),
    \qquad
    R\geq\|x_0-x^\ast\|_p.
\]
Substitution into \cref{eq:mirror_phase_rate} gives the following
two-phase guarantee. The complete procedure appears in
\cref{alg:two-phase-mirror-dual}.

\begin{theorem}[Inexact mirror-dual rate]
\label{thm:mirror_final_rate}\linktoproof{thm:mirror_final_rate}
Suppose \cref{ass:holder-setting} holds with \(p\in[2,\infty)\). Let \(x_0\in\mathbb{R}^d\), \(\epsilon>0\), and
\(R\geq\|x_0-x^\ast\|_p\). \cref{alg:two-phase-mirror-dual} it returns a point \(x\)
satisfying \(\|\nabla f(x)\|_{p^\ast}\leq\epsilon\), if \(\kappa=2\), using
\[
    O_{ p}\!\left(
        \left(\frac{LR}{\epsilon}\right)^{ p/( p+2)}
    \right)
\]
gradient evaluations. If \(1<\kappa<2\) and
\(
    D_{ p,\kappa}
    \defi
    \kappa p+\kappa
    -\left(\frac12+\frac1\kappa\right) p
    >0,
\)
then the gradient evaluations are
\[
    O_{ p,\kappa}\!\left(
        \left(
            \frac{LR^{\kappa-1}}{\epsilon}
        \right)^{ p/D_{ p,\kappa}}
    \right).
\]
\end{theorem}

The prior knowledge required for tuning can be removed, up to logarithmic factors, by
standard doubling and geometric-search arguments over the horizon,
smoothness scale, and internal tuning parameters, using the computable certificate
$\|\nabla f(x)\|_{p^\ast}\leq\epsilon$ for termination. See, e.g.,
\citet{lan2023optimal} for related parameter-free constructions.

This inexact mirror
duality method achieves optimality for smooth functions ($\kappa=2$) and has a simpler iteration than our next method, but does not close the
Hölder-smooth complexity gap, even though it improves over previous work in several cases, see \cref{prop:rate-comparison} for a rate comparison. The method in the next section, however, is optimal
up to logarithmic factors across the range $p\in[1,\infty)$.

\section{Accumulative Power Regularization}
\label{sec:accumulative_regularization}

We now develop an accumulative-regularization scheme adapted to the
$\ell_p$ geometry. We make use of the shifted-power regularizers 
\(
r_q(x):=\frac{1}{q}\|x\|_p^q, q:=\max\{2,p\},
\)
introduced in \cref{subsec:powerregshifted}.
The endpoint cases $p=1$ and $p=\infty$ are handled through the proxy
norms in \cref{rem:endpoint-norms}.

We run an algorithm in stages. At each stage, the objective is regularized from the current approximate solution, on top of regularization set at previous stages. Then, we minimize the resulting object to a prescribed function-value accuracy using Generalized \(\mathrm{AGD}^{+}\)~\citep{diakonikolas2024complementary}. The shifted-power subproblems introduced in \cref{eq:shifted_power_call} are required. We show that an appropriate geometric choice of regularization weights and target accuracies controls the distance between successive minimizers and transfers the final function-value guarantee into a bound on \(\|\nabla f\|_{p^*}\). Finally, accounting for the complexity of each inner iteration to its required accuracy, we obtain a complexity that nearly matches the lower bound in \citet{diakonikolas2024complementary}. Our complexity does not contain any logarithmic factors except for $p=1$ and $p=\infty$, and except maybe for these two cases, we believe that the remaining logarithmic factors between the upper and the lower bounds are due to the lower bounds not being fully optimal.

Throughout the analysis in this subsection, we work
under \cref{ass:holder-setting} with \(1<p<\infty\) and set
\(q := \max\{2,p\}.\)

At an outer stage, we minimize $H_s(x) := f(x)+\Psi_s(x)$, where
$f$ is our objective and $\Psi_s(x)=\sum_{i=1}^m \alpha_i r_q(x-c_i)$ is our accumulative regularization term, for some weights $\alpha_i = \sigma_i - \sigma_{i-1}\geq0$, where
$\sigma_i > 0$ is a parameter that grows exponentially, and we let
$\lambda_s=\mu_{p,q}\sigma_s$ be the uniformly convex constant of $H_s$.

Given an initial point $\bar y$, we use the stage-centered auxiliary regularizer
$\phi_{\bar y}(x)\defi\mu_{p,q}^{-1}r_q(x-\bar y)
=(q\mu_{p,q})^{-1}\|x-\bar y\|_p^q$.
By the uniform convexity of $r_q$,
$\phi_{\bar y}(x)\geq q^{-1}\|x-\bar y\|_p^q$ and
$D_{\phi_{\bar y}}(x,z)\geq q^{-1}\|x-z\|_p^q$.
So $\phi_{\bar y}$ is an admissible auxiliary regularizer for the Generalized $\AGD$
framework of \citet[Theorem~2]{diakonikolas2024complementary}.
The step sizes $a_k$, accumulated weights $A_k$, parameters $M_k$ and
$m_0$, and inexactness schedule are chosen exactly as in that theorem.

\begin{fact}[AGD+ Inner complexity]
\label{fact:generalized-agd-inner}
    Let $x_s^\ast\in\argmin \{H_s := f+\Psi\}$ where $f$ is $(L,\kappa)$-Hölder smooth and convex and $\Psi$ is $(\mu_{p,q}, q)$-uniformly convex for $q = \max\{p, 2\}$. Suppose
$R\geq\|x_s^\ast-\bar y\|_p$. With the parameter choices of
\citet[Theorem~2]{diakonikolas2024complementary},
\cref{alg:generalized-agd} (AGD+) returns $y_N$ satisfying
$H_s(y_N)-H_s(x_s^\ast)\leq\delta$ after
$O_{p,\kappa}\!\left(1+
(LR^\kappa/\delta)^{1/\tau}\right)$ iterations, where
$\tau=((q+1)\kappa-q)/q$.
\end{fact}

Indeed, the second argument in the $\min$ 
\citet[Theorem~2]{diakonikolas2024complementary} gives a rate
$$O_{q,\kappa}\!\left(
1+(L/\delta)^{q/((q+1)\kappa-q)}
\phi_{\bar y}(x_s^\ast)^{\kappa/((q+1)\kappa-q)}
\right),$$
for optimizing $H_s$ up to accuracy $\delta$.
Since \(q=\max\{2,p\}\), constants depending on \(q\) are equivalently
absorbed into the \(O_{p,\kappa}(\cdot)\) notation used in the statement.
Our choice satisfies $\phi_{\bar y}(x_s^\ast)
=(q\mu_{p,q})^{-1}\|x_s^\ast-\bar y\|_p^q
\leq R^q/(q\mu_{p,q})$
and $q/((q+1)\kappa-q) = 1/\tau$, from which the claimed bound follows.
Each iteration evaluates $\nabla f$ once. The use of this fast subroutine is essential for the final complexity of our algorithm.

\begin{algorithm}[t]
\caption{Accumulative Regularization for Hölder-smooth convex objectives}
\label{alg:holder-ar}
\footnotesize
\begin{algorithmic}[1]
    \Require Differentiable convex $(L, \kappa)$-Hölder smooth function $f$ in $\norm{\cdot}_p$, initial point $x_0$, bound on initial distance to a minimizer $R_0$, accuracy $\eps$. Regularizer $r_q$ with constant $\mu_{p,q}$, as in \cref{eq:shifted_power_call}
    \vspace{0.1cm}
    \hrule
    \vspace{0.1cm}
\State $q=\max\{2,p\}$,
    \State $\sigma_0=0$, $\sigma_1= \frac{(1-2^{-(\kappa-1)/2})\eps} {2\cdot3^{q-1}R_0^{q-1}}$, $\sigma_s = \sigma_12^{(s-1)(q-(\kappa+1)/2)}$

    \For{$s=1,\ldots, \left\lceil \frac{1}{\kappa-1}\log_+(\frac{2LR_0^{\kappa-1}}{\eps}) \right\rceil$}
    \State $R_s \gets R_{s-1} / 2$
    \quad
    $\delta_s=\mu_{p,q}\sigma_s R_s^q/q$
    \State $\Psi_s(x) = \sum_{i=1}^s (\sigma_i - \sigma_{i-1}) r_q(x-x_{i-1})$ ,
    \State $\phi_s(x)=
    \mu_{p,q}^{-1}r_q(x-x_{s-1})$
    \State $x_s \gets \AGDp(x_{s-1}, \delta_s, R_sis, f,\Psi_s, \phi_s)$
\EndFor

\State \Return $x_S$

\end{algorithmic}
\end{algorithm}

We use the following observation for the coverage of the cases $p=1$ and $p=\infty$ by our algorithm.
\begin{remark}[\(p=1\) and $p=\infty$]
\label{rem:endpoint-norms}
    For \(\widehat p=1+1/\log d\), it is $\|x\|_1 = \Theta(\|x\|_{\widehat p})$. Similarly, for \(\widehat p=\log d\), we have $\|x\|_{\infty} = \Theta(\|x\|_{\widehat p})$. Thus, a problem in $\norm{\cdot}_1$ or $\norm{\cdot}_\infty$ can be phrased as a problem in $\ell_{\widehat p}$, for those values of $\widehat{p}$, up to universal constant changes in problem parameters. This amounts to adding a dependence on $\log(d)$ on our rates for these two extreme points due to the dependence of $p$ in the $O_p(\cdot)$ complexity.
\end{remark}

Our complete method is given in \cref{alg:holder-ar}. We now present its guarantees.

\begin{theorem}[H\"older gradient-norm minimization]
\label{thm:accumulative-power}\linktoproof{thm:accumulative-power}
Under \cref{ass:holder-setting}, let
\(R_0\ge\|x_0-x^\ast\|_p\) for \(p\in[1,\infty)\).
Using \cref{alg:holder-ar}, we can obtain a point \(x_S\) such that 
\(
    \|\grad f(x_S)\|_{p^\ast}\le\eps
\)
using
\(
    O_{p,\kappa}\!\left(
        \left(\frac{L R_0^{\kappa-1}}{\eps}\right)^{
            q /((q+1)\kappa-q)}
    \right)
\)
gradient evaluations, where $q = \max\{p, 2\}$.
\end{theorem}

We refer to \cref{rem:endpoint-norms} for the rate when $p=\infty$.

\begin{figure*}[t]
    \centering
    \includegraphics[width=\textwidth]{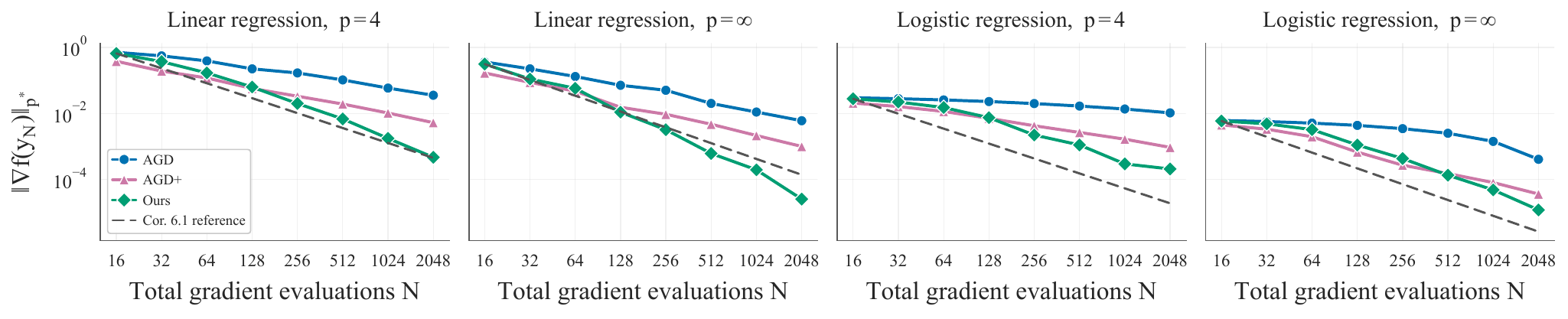}
    \caption{Dual-norm gradient convergence on $\ell_p$-norm for problems specified in \cref{sec:regression_problems} at $p\in\{4,\infty\}$ (with $p=\infty$ handled via the proxy $\widehat p=\log d$). Each panel plots $\norm{\nabla f(y_N)}_{p^\ast}$ versus the number of gradient evaluations $N$ for \AGD{}, AGD+~\citep{diakonikolas2024complementary} and our mirror-dual method. The dashed line is the smooth rate $N^{-(1+2/p)}$ from \cref{thm:mirror_final_rate}.}
    \label{fig:methods_comparison}
\end{figure*}

\cref{thm:accumulative-power} closes the aforementioned H\"older-smooth convex open problem on gradient minimization, which to the best of our knowledge was open even for the Euclidean case. Its current oracle model is more complex per iteration, cf. \cref{eq:shifted_power_call}, than that of the mirror-dual approach, while only requiring one gradient oracle call as well.

\section{Numerical Illustration}\label{sec:experiments}
We compare our mirror-dual method (\textit{\texttt{Ours}}) to two baselines: primal \AGD{} on the unregularised objective, and the regularisation-based gradient-norm minimiser of \citet{diakonikolas2024complementary} (denoted \textit{AGD+}). We consider $\ell_p$-norm linear regression $f(x)=\norm{Ax-b}_p^2$ and logistic regression $F(x)=\frac1n\sum_{i=1}^n\log(1+\exp(-y_i\langle a_i,x\rangle))$ at $p=4$ and $p=\infty$ (handled through the proxy $\widehat p=\log d$ of \cref{rem:endpoint-norms}). Data $a_i\in\R^d$ are drawn uniformly from the unit $p^\ast$-ball, so $\norm{a_i}_{p^\ast}\le1$ holds by construction. For linear regression $A\in\R^{60\times 30}$ and $b=Ax_{\text{true}}+0.05\eta$ with $\eta\sim\mathcal{N}(0,I_{60})$, while for logistic regression $n=200$ and $d=30$. Our objectives satisfy the assumptions in this work, as we show in 
\cref{prop:linear-regression-smoothness,prop:logistic-regression-smoothness}.

\Cref{fig:methods_comparison} reports the dual-norm gradient $\norm{\nabla f(y_N)}_{p^\ast}$ as a function of the gradient-evaluation budget $N\in\{16,32,\ldots,2048\}$. The dashed line in each panel is the smooth rate $N^{-(1+2/p)}$ predicted by \cref{thm:mirror_final_rate}. Across all four panels \texttt{Ours} strictly improves on AGD+: the empirical exponent gap matches the theoretical $2(p-1)/p$, and at $N=2048$ \texttt{Ours} ends roughly an order of magnitude below AGD+ in every panel, with the gap most pronounced at $p=\infty$, the endpoint regime handled through \cref{rem:endpoint-norms}. Logistic regression is in a pre-asymptotic regime (the small Lipschitz constant $L=\nicefrac14$ inflates the multiplicative constant in front of the rate), so the absolute slopes there have not yet reached the theoretical exponent. The ordering of the methods is nonetheless preserved.

\section{Conclusion}\label{sec:conclusion}

We studied the problem of obtaining verifiable small-gradient
certificates for H\"older-smooth convex optimization in
non-Euclidean geometry. Two complementary mechanisms emerge from our
analysis. Inexact mirror duality extends the time-reversal principle
behind accelerated gradient minimization beyond quadratic smoothness
and, at the smooth endpoint, closes the previously unresolved regime
\(p>2\). Accumulative power regularization instead exploits the
underlying geometry through the choice \(q=\max\{2,p\}\): successive
function-value guarantees localize the regularized minimizers, and this
localization transfers to a gradient bound for the original objective.

Taken together, these results characterize the optimal exponent of the
natural accuracy scale \(LR^{\kappa-1}/\eps\) for gradient minimization
throughout \(p \in [1, \infty)\) and \(1<\kappa\le2\). Thus, at the level of
gradient evaluations, the Hölder-smooth non-Euclidean complexity gap is
closed. The principal remaining issue is algorithmic: replacing the
exact multi-center shifted-power calls required by accumulative
regularization with efficiently implementable approximate solves
without sacrificing the near-optimal rate.

Future work includes replacing the multi-center shifted-power subproblem with approximate solves that preserve the overall oracle complexity, and extending the theory to constrained domains.

\bibliography{refs}
\appendix
\onecolumn

\section{Proofs of Preliminaries}

\begin{proof}[\linkofproof{fact:power_geometry}]
Let \(J_p(u):=(|u_i|^{p-2}u_i)_{i=1}^d\). For \(1<p\le2\), we have
\(q=2\) and
\(
\grad r_2(u)=\|u\|_p^{2-p}J_p(u).
\)
The normalized duality map satisfies
\(
\langle \grad r_2(x)-\grad r_2(y),x-y\rangle
\ge (p-1)\|x-y\|_p^2 .
\)
Writing \(d=y-x\), the integral representation of the Bregman divergence
gives
\[
D_{r_2}(y,x)
=\int_0^1\langle \grad r_2(x+td)-\grad r_2(x),d\rangle\,dt
\ge (p-1)\int_0^1 t\|d\|_p^2\,dt
=\frac{p-1}{2}\|d\|_p^2 .
\]

For \(p\ge2\), we have \(q=p\) and \(\grad r_p=J_p\). Using
\(
(|a|^{p-2}a-|b|^{p-2}b)(a-b)\ge 2^{2-p}|a-b|^p
\)
coordinatewise yields
\[
D_{r_p}(y,x)
=\int_0^1\langle J_p(x+td)-J_p(x),d\rangle\,dt
\ge \int_0^1 2^{2-p}t^{p-1}\|d\|_p^p\,dt
=\frac{2^{2-p}}{p}\|d\|_p^p .
\]
Hence \(r_q\) is \((\mu_{p,q},q)\)-uniformly convex with the claimed
constant.

Moreover,
\(
\grad r_q(u)=\|u\|_p^{q-p}J_p(u),
\)
and therefore
\(
\|\grad r_q(u)\|_{p^*}
=\|u\|_p^{q-p}\|J_p(u)\|_{p^*}
=\|u\|_p^{q-1},
\)
where we used \(\|J_p(u)\|_{p^*}=\|u\|_p^{p-1}\).
\end{proof}

\begin{proof}[\linkofproof{lemma:holder_inexact}]
By H\"older smoothness and Young's inequality, for every $x,y$ and
$\gamma>0$,
\begin{align*}
D_f(y,x)
&\leq \frac{L}{\kappa}\|y-x\|^\kappa
\leq \frac{L}{2\gamma^{2/\kappa}}\|y-x\|^2
+ L\gamma^{2/(2-\kappa)}\frac{2-\kappa}{2\kappa}.
\end{align*}
This is precisely $\delta$-inexact $\widetilde L$-smoothness with the
parameters stated in the lemma.
\end{proof}

\begin{proof}[\linkofproof{prop:conjugate_holder}]
Applying uniform convexity with $(x,y)$ and $(y,x)$ and adding gives
\[
\innp{\nabla\psi(x)-\nabla\psi(y),x-y}
\geq \frac{2\mu}{p}\|x-y\|^p.
\]
Uniform convexity makes \(\psi\) coercive, so for every \(u\) the
function \(x\mapsto\psi(x)-\langle u,x\rangle\) has a unique minimizer.
Hence \(\grad\psi\) is onto and
\(\grad\psi^*=(\grad\psi)^{-1}\).
Let $u=\nabla\psi(x)$ and $v=\nabla\psi(y)$. By H\"older's inequality,
\[
\|x-y\|^{p-1}
\leq \frac{p}{2\mu}\|u-v\|_*,
\]
and hence
\[
\|\nabla\psi^*(u)-\nabla\psi^*(v)\|
\leq
\left(\frac{p}{2\mu}\right)^{1/(p-1)}
\|u-v\|_*^{1/(p-1)}.
\]
Since \(1/(p-1)=p^*-1\), we have
\[
    \|\grad\psi^*(u)-\grad\psi^*(v)\|
    \leq
    M\|u-v\|_*^{p^*-1},
    \qquad
    M=\left(\frac{p}{2\mu}\right)^{p^*-1}.
\]
Integrating this inequality along the line segment from \(v\) to \(u\)
gives
\[
\begin{aligned}
    D_{\psi^*}(u,v)
    &=
    \int_0^1
    \left\langle
        \grad\psi^*(v+t(u-v))-\grad\psi^*(v),
        u-v
    \right\rangle dt \\
    &\leq
    M\|u-v\|_*^{p^*}
    \int_0^1 t^{p^*-1}\,dt
    =
    \frac{M}{p^*}\|u-v\|_*^{p^*}.
\end{aligned}
\]
Thus, in the sense of
\cref{def:uniform-convexity-holder-smoothness},
\(\psi^*\) is \((M,p^*)\)-H\"older smooth.
\end{proof}

\section{Proofs of Inexact Mirror Duality}

We give the details behind the two estimates summarized in
\cref{sec:mirror_duality}. The proof follows the inexact mirror-duality
argument used in the earlier version of this paper.

\begin{lemma}[Inexact cocoercivity]
\label{lem:inexact_cocoercivity}
If $\varphi$ is convex and $\delta$-inexact $K$-smooth with respect to
$\|\cdot\|$, then
\[
D_\varphi(x,y)
\geq
\frac{1}{2K}\|\nabla\varphi(x)-\nabla\varphi(y)\|_*^2-\delta.
\]
\end{lemma}

\begin{proof}
Fix $y$ and define
$\Phi(x)=\varphi(x)-\innp{\nabla\varphi(y),x}$. Then $y$ minimizes
$\Phi$. Inexact smoothness gives, for every $z$,
\[
\Phi(z)
\leq
\Phi(x)+\innp{\nabla\Phi(x),z-x}
+\frac{K}{2}\|z-x\|^2+\delta.
\]
Minimizing the right-hand side over $z$ and using Fenchel duality yields
\[
\Phi(y)
\leq
\Phi(x)-\frac{1}{2K}\|\nabla\Phi(x)\|_*^2+\delta.
\]
The result follows from
$\nabla\Phi(x)=\nabla\varphi(x)-\nabla\varphi(y)$ and
$\Phi(x)-\Phi(y)=D_\varphi(x,y)$.
\end{proof}

\subsection{Generalized Mirror Duality}

A coupled first-order method (CFOM) has the general form
\begin{align*}
y_{k+1} &= y_k - \sum_{j=0}^{k} \alpha_{k+1,j} \nabla f(x_j), \\
x_{k+1} &= x_k - \sum_{j=0}^{k+1} \beta_{k+1,j} \nabla \psi^*(y_j).
\end{align*}
Let $\mathcal A$ be a primal CFOM and let $\mathcal B$ be its mirror
dual, obtained by time-reversing the two coefficient arrays. For fixed
quadratic model constants $\widehat L,\widehat M>0$, initialize the
primal potential at
\[
    U_0\defi D_\psi(x^*,x_0)
\]
and define
\begin{align*}
U_{t+1}
&\defi U_t-a_{t+1}D_f(x^*,x_{t+1})\\
&\quad-A_t\left(
D_f(x_t,x_{t+1})
-\frac1{2\widehat L}
\|\nabla f(x_t)-\nabla f(x_{t+1})\|_*^2
\right)\\
&\quad-\left(
D_{\psi^*}(y_t,y_{t+1})
-\frac1{2\widehat M}
\|\nabla\psi^*(y_{t+1})-\nabla\psi^*(y_t)\|^2
\right)
\end{align*}
for $t=0,\ldots,T-1$. Collecting the terms in the terminal expansion
defines the primal residual functional $\mathbf U_{\mathcal A}$ by
\begin{equation}
\label{eq:appendix_primal_residual}
U_T
=A_T\bigl(f(x_T)-f(x^*)\bigr)
+\psi(x^*)+\psi^*(y_T)-\innp{y_T,x^*}
+\mathbf U_{\mathcal A}.
\end{equation}

Initialize the time-reversed dual potential at
\[
    V_0\defi\frac{f(q_0)-f(x^*)}{A_T}.
\]
Applying the same interpolation chain after reversing the coefficient
arrays defines its terminal residual $\mathbf V_{\mathcal B}$ through
\begin{equation}
\label{eq:appendix_dual_residual}
V_T
=\psi^*(r_T)+D_{\psi^*}(0,r_0)
+\mathbf V_{\mathcal B}.
\end{equation}
Both residuals are quadratic functionals of the oracle values generated
by their respective CFOMs.

\begin{theorem}[Generalized mirror duality
{\citep[Theorem~1]{kim2023mirror}}]
\label{thm:appendix_generalized_mirror_duality}
With the definitions above,
\[
    \inf_{x_0,\ldots,x_T,\,y_0,\ldots,y_T}
    \mathbf U_{\mathcal A}
    =
    \inf_{q_0,\ldots,q_T,\,r_0,\ldots,r_T}
    \mathbf V_{\mathcal B}.
\]
In particular, $\mathbf U_{\mathcal A}\geq0$ for all primal oracle
values implies $\mathbf V_{\mathcal B}\geq0$ for all dual oracle
values.
\end{theorem}

\begin{proof}
The identity is algebraic: time reversal permutes the oracle values and
transposes the two coefficient arrays in the residual quadratic form.
It therefore does not require exact smoothness of $f$ or exact strong
convexity of $\psi$. The equality is precisely
\citet[Theorem~1]{kim2023mirror} applied to
$\mathcal A$ and its mirror dual $\mathcal B$.
\end{proof}

For the remainder of this section, set
\[
\psi(u)=\frac1p\|u\|_p^p,
\qquad
M=\left(\frac{p}{2\mu_p}\right)^{p^*-1},
\qquad
\mu_p=2^{2-p}.
\]
Thus $M=O_p(1)$ and
$\psi^*(z)=\frac1{p^*}\|z\|_{p^*}^{p^*}$. Introduce the inexact
smoothness parameters
\[
\begin{gathered}
\widehat M=
\begin{cases}
M, & p^*=2,\\
M\gamma_1^{-2/p^*}, & p^*<2,
\end{cases}
\qquad
\delta_\psi=
\begin{cases}
0, & p^*=2,\\
\displaystyle
\frac{M(2-p^*)}{2p^*}\gamma_1^{2/(2-p^*)}, & p^*<2,
\end{cases}
\\[1ex]
\widehat L=
\begin{cases}
L, & \kappa=2,\\
L\gamma_2^{-2/\kappa}, & \kappa<2,
\end{cases}
\qquad
\delta_f=
\begin{cases}
0, & \kappa=2,\\
\displaystyle
\frac{L(2-\kappa)}{2\kappa}
\gamma_2^{2/(2-\kappa)}, & \kappa<2.
\end{cases}
\end{gathered}
\]
These expressions follow from \cref{lemma:holder_inexact}, applied to
$f$ and $\psi^*$. We use the accelerated weights
\begin{equation}
\label{eq:appendix_mirror_weights}
a_t=\frac{t}{2\widehat L\widehat M},
\qquad
A_t=\sum_{i=1}^t a_i
=\frac{t(t+1)}{4\widehat L\widehat M}.
\end{equation}
They satisfy
\begin{equation}
\label{eq:appendix_weight_psd}
a_t^2\leq\frac{A_t}{\widehat L\widehat M}.
\end{equation}

\begin{lemma}[Power tradeoff]
\label{lem:appendix_power_tradeoff}
Let $A,B,u,v>0$. Then the function
$h(s)=As^u+Bs^{-v}$ has the unique minimizer
\[
    s^*=\left(\frac{vB}{uA}\right)^{1/(u+v)}.
\]
Moreover,
\[
    h(s^*)
    =A^{v/(u+v)}B^{u/(u+v)}
      \left[
        \left(\frac vu\right)^{u/(u+v)}
        +\left(\frac uv\right)^{v/(u+v)}
      \right].
\]
\end{lemma}

\begin{proof}
Differentiate $h$, solve $h'(s)=0$, and substitute the resulting
$s^*$ back into $h$. Strict convexity after the logarithmic change of
variables $s=e^z$ gives uniqueness.
\end{proof}

\begin{lemma}[Primal warm-start rate]
\label{lem:appendix_primal_warm_start}
The accelerated primal CFOM underlying
$\mathsf{MD}(\mathsf{AGD})$, initialized at $x_0$, can be tuned so that
after $T$ gradient evaluations its output $x_T$ satisfies
\[
f(x_T)-f(x^*)
=O_{p,\kappa}\!\left(
\frac{LR^\kappa}
{T^{((p+1)\kappa-p)/p}}
\right),
\qquad
R\geq\|x_0-x^*\|_p.
\]
\end{lemma}

\begin{proof}
Use the centered regularizer
$\psi_0(x)=\frac1p\|x-x_0\|_p^p$ and the weights in
\cref{eq:appendix_mirror_weights}. The accelerated estimate-sequence
expansion, with \cref{lem:inexact_cocoercivity} applied to $f$ and
$\psi_0^*$, gives
\begin{equation}
\label{eq:appendix_primal_potential}
f(x_T)-f(x^*)
\leq
\frac{
D_{\psi_0}(x^*,x_0)+T\delta_\psi
+\delta_f\sum_{t=1}^{T-1}A_t
}{A_T}.
\end{equation}
For completeness, the only residual term in this expansion is a sum of
quadratic forms of the type
\[
\frac{A_t}{2\widehat L}\|g_t-g_{t+1}\|_*^2
+\frac{1}{2\widehat M}\|z_t-z_{t+1}\|^2
+a_t\innp{g_t-g_{t+1},z_{t+1}-z_t}.
\]
Their sum is the residual $\mathbf U_{\mathcal A}$ in
\eqref{eq:appendix_primal_residual}. Each form is nonnegative by
H\"older's inequality and \cref{eq:appendix_weight_psd}. Hence
$\mathbf U_{\mathcal A}\geq0$, which proves
\cref{eq:appendix_primal_potential}.

Set $D\defi R^p/p\geq D_{\psi_0}(x^*,x_0)$ and denote the
right-hand side of \eqref{eq:appendix_primal_potential}, with
$D_{\psi_0}(x^*,x_0)$ replaced by $D$, by
\[
    h_{\mathrm P}(\gamma_1,\gamma_2)
    \defi
    \frac{D+T\delta_\psi(\gamma_1)
    +\delta_f(\gamma_2)\sum_{t=1}^{T-1}A_t}{A_T}.
\]
Thus $f(x_T)-f(x^*)\leq h_{\mathrm P}(\gamma_1,\gamma_2)$.

Suppose first that $p>2$ and $1<\kappa<2$. Define
\[
\alpha\defi\frac{2}{p^*},
\quad
r\defi\frac{2}{2-p^*}-\frac{2}{p^*},
\quad
C_1\defi\frac{4LMD}{T(T+1)},
\quad
C_2\defi\frac{2LM^2(2-p^*)}{p^*(T+1)},
\]
and $C_3\defi L(2-\kappa)(T-1)/(6\kappa)$. Direct substitution gives
\[
 h_{\mathrm P}(\gamma_1,\gamma_2)
 =\gamma_2^{-2/\kappa}
  \left(C_1\gamma_1^{-\alpha}+C_2\gamma_1^r\right)
 +C_3\gamma_2^{2/(2-\kappa)}.
\]
By \cref{lem:appendix_power_tradeoff}, the first tradeoff is minimized
at
\begin{equation}
\label{eq:appendix_primal_gamma1}
\gamma_{1,\mathrm P}
=
\left(\frac{\alpha C_1}{rC_2}\right)^{1/(\alpha+r)}
=
\left(\frac{pD}{TM}\right)^{(2-p^*)/2}.
\end{equation}
Let
\[
G_{\mathrm P}
\defi
C_1\gamma_{1,\mathrm P}^{-\alpha}
+C_2\gamma_{1,\mathrm P}^{r}.
\]
The second tradeoff is minimized at
\begin{equation}
\label{eq:appendix_primal_gamma2}
\gamma_{2,\mathrm P}
=
\left(\frac{6G_{\mathrm P}}{L(T-1)}\right)^{\kappa(2-\kappa)/4}.
\end{equation}
If $p=2$, the $\gamma_1$ tradeoff disappears and we use the same
$\gamma_{2,\mathrm P}$ after setting
$G_{\mathrm P}=4LMD/[T(T+1)]$. If $\kappa=2$, the $\gamma_2$
tradeoff disappears and only \eqref{eq:appendix_primal_gamma1} is used.
If $p=\kappa=2$, neither parameter is needed. In all four cases,
substitution gives
\[
f(x_T)-f(x^*)
=O_{p,\kappa}\!\left(
\frac{L M^{\kappa/p^*}R^\kappa}
{T^{2\kappa-\kappa/p^*-1}}
\right).
\]
Since $M=O_p(1)$ and
$2\kappa-\kappa/p^*-1=((p+1)\kappa-p)/p$, the claim follows.
\end{proof}

\begin{lemma}[Mirror-dual phase]
\label{lem:appendix_mirror_phase}
Let $\Delta=f(q_0)-f(x^*)$. The mirror dual of the accelerated primal
CFOM can be tuned so that its final iterate satisfies
\[
\|\nabla f(q_T)\|_{p^*}
=O_{p,\kappa}\!\left(
\frac{L^{1/\kappa}\Delta^{(\kappa-1)/\kappa}}
{T^{3/2+1/p-2/\kappa}}
\right)
\]
when $1<\kappa<2$. When $\kappa=2$, it satisfies
\[
\|\nabla f(q_T)\|_{p^*}
=O_p\!\left(
\frac{\sqrt{L\Delta}}{T^{1/2+1/p}}
\right).
\]
\end{lemma}

\begin{proof}
The primal analysis above establishes
$\mathbf U_{\mathcal A}\geq0$. By
\cref{thm:appendix_generalized_mirror_duality}, the mirror-dual
residual satisfies $\mathbf V_{\mathcal B}\geq0$. Applying
\cref{lem:inexact_cocoercivity} to the remaining remainders gives
\begin{equation}
\label{eq:appendix_dual_master}
\psi^*(r_T)
\leq
\frac{\Delta}{A_T}
+T\delta_\psi
+\delta_f\sum_{t=1}^{T-1}\frac1{A_t}.
\end{equation}
By \cref{prop:appendix-gradient-hitting}, the final dual iterate is
\begin{equation}
\label{eq:appendix_gradient_hitting}
r_T=\nabla f(q_T).
\end{equation}

Denote the right-hand side of \eqref{eq:appendix_dual_master} by
\[
    h_{\mathrm D}(\gamma_1,\gamma_2)
    \defi
    \frac{\Delta}{A_T}
    +T\delta_\psi(\gamma_1)
    +\delta_f(\gamma_2)\sum_{t=1}^{T-1}\frac1{A_t}.
\]
Suppose first that $p>2$ and $1<\kappa<2$. Define
\[
\beta\defi\frac2\kappa,
\quad
\eta\defi\frac{4(\kappa-1)}{\kappa(2-\kappa)},
\quad
B_0\defi\frac{4LM\Delta}{T(T+1)},
\quad
C_0\defi\frac{2L^2M(2-\kappa)(T-1)}{\kappa T},
\]
and
\[
u\defi\frac{2}{2-p^*},
\qquad
v\defi\frac{2}{p^*},
\qquad
C_\psi\defi\frac{M(2-p^*)T}{2p^*}.
\]
Then
\[
 h_{\mathrm D}(\gamma_1,\gamma_2)
 =C_\psi\gamma_1^u
 +\gamma_1^{-v}
  \left(B_0\gamma_2^{-\beta}+C_0\gamma_2^\eta\right).
\]
The inner tradeoff is minimized at
\begin{equation}
\label{eq:appendix_dual_gamma2}
\gamma_{2,\mathrm D}
=
\left(\frac{\beta B_0}{\eta C_0}\right)^{1/(\beta+\eta)}.
\end{equation}
Let
\[
G_{\mathrm D}
\defi
B_0\gamma_{2,\mathrm D}^{-\beta}
+C_0\gamma_{2,\mathrm D}^{\eta}.
\]
The remaining tradeoff is minimized at
\begin{equation}
\label{eq:appendix_dual_gamma1}
\gamma_{1,\mathrm D}
=
\left(\frac{vG_{\mathrm D}}{uC_\psi}\right)^{1/(u+v)}.
\end{equation}
If $p=2$, the $\gamma_1$ tradeoff disappears. If $\kappa=2$, the
$\gamma_2$ tradeoff disappears and we set
$G_{\mathrm D}=4LM\Delta/[T(T+1)]$. If both are equal to two, neither
parameter is needed. Substitution gives
\[
\psi^*(r_T)
=O_{p,\kappa}\!\left[
\left(
\frac{L^{1/\kappa}M^{1/p^*}\Delta^{(\kappa-1)/\kappa}}
{T^{3/2+1/p-2/\kappa}}
\right)^{p^*}
\right]
\]
for $1<\kappa<2$, and
\[
\psi^*(r_T)
=O_p\!\left[
\left(
\frac{\sqrt{L\Delta}\,M^{1/p^*}}
{T^{1/2+1/p}}
\right)^{p^*}
\right]
\]
when $\kappa=2$. Since
$\psi^*(r_T)=\|r_T\|_{p^*}^{p^*}/p^*$, $M=O_p(1)$, and
\cref{eq:appendix_gradient_hitting} holds, these are the stated rates.
\end{proof}

\begin{proof}[\linkofproof{thm:mirror_final_rate}]
The corresponding geometry in \cref{alg:two-phase-mirror-dual} is the
    \(\ell_{p}\) geometry for $p \in (1, \infty)$, but we can use \cref{rem:endpoint-norms}.

Run the primal phase of \cref{lem:appendix_primal_warm_start} for $T$
iterations and initialize the mirror-dual phase of
\cref{lem:appendix_mirror_phase} at its output. In the smooth case,
\[
\Delta
=O_p\!\left(\frac{LR^2}{T^{(p+2)/p}}\right),
\]
and therefore
\[
\|\nabla f(q_T)\|_{p^*}
=O_p\!\left(\frac{LR}{T^{(p+2)/p}}\right).
\]
Choosing
$T=O_p((LR/\epsilon)^{p/(p+2)})$ proves the first claim.

For $1<\kappa<2$, the same substitution gives
\begin{align*}
\|\nabla f(q_T)\|_{p^*}
&=O_{p,\kappa}\!\left(
\frac{LR^{\kappa-1}}{T^\Gamma}
\right),\\
\Gamma
&=\left(\frac32+\frac1p-\frac2\kappa\right)
+\frac{\kappa-1}{\kappa}
\left(\frac{(p+1)\kappa-p}{p}\right)\\
&=\frac{\kappa p+\kappa-(\frac12+\frac1\kappa)p}{p}
=\frac{D_{p,\kappa}}{p}.
\end{align*}
When $D_{p,\kappa}>0$, choosing
\[
T=O_{p,\kappa}\!\left(
\left(\frac{LR^{\kappa-1}}{\epsilon}\right)^{p/D_{p,\kappa}}
\right)
\]
proves the second claim. Each phase uses $T$ gradient evaluations, so
the factor of two is absorbed by the $O_{p,\kappa}$ notation.
\end{proof}

\begin{remark}\label{prop:rate-comparison}
We compare the exponents of $LR^{\kappa-1} / \epsilon$ in the expressions of the lower bound $a_{\rm L}$, of the upper bound from \citet{diakonikolas2024complementary} $a_{\rm P}$ and of the upper bound from our algorithm $a_{\rm O}$, all under \cref{ass:holder-setting}, that is, the Hölder smoothness assumption. These exponents are:
\begin{align*}
    a_{\rm L}= \frac{p}{\kappa p+\kappa-p}, \quad
    a_{\rm P}=\frac{\kappa(p-1)}{(\kappa-1)(\kappa p+\kappa-p)}, \quad
    a_{\rm O}= \frac{p}{\kappa p+\kappa-(\frac12+\frac1\kappa)p},
\end{align*}
One can observe that our rate does not violate the lower bound since $\frac{1}{2}+\frac{1}{\kappa} \geq 1$ due to $\kappa \in (1, 2)$. Also, our rates are non-vacuous whenever the denominator of $a_{\rm O}$ is positive. This occurs for all $p \geq 2$ for all $\kappa \geq \kappa_\ast$, with $\kappa_\ast=\frac{1+\sqrt{17}}{4}\approx1.28$ as a solution of a simple quadratic equation. If $\kappa<\kappa_\ast$ then we have a positive denominator if 
$
  p<\frac{2\kappa^2}{\kappa+2-2\kappa^2}.
$

Finally, we compare against \citet{diakonikolas2024complementary}. The inequality $a_{\rm O}<a_{\rm P}$ is equivalent to the quadratic inequality
$
    (3\kappa-4)p^2+(-2\kappa^2+3\kappa+2)p-2\kappa^2>0.
$
If $\frac43\leq\kappa\le 2$, then it is a convex quadratic which is negative at $p=0$ and positive at $p=2$ and thus, it is positive for all $p\geq 2$. In other words, our rates are better in this regime.
If $1<\kappa<\frac43$, the inequality is a concave quadratic and our rates are better between its roots $p_{\pm}(\kappa)$ (closed-form in the appendix), intersected with the positive denominator condition. Since $p_{-}(\kappa)\in(0,2)$, we are better for moderate $p$ from $2$ to $p_{+}(\kappa)$.
\end{remark}

\section{Additional Algorithms}\label{sec:additional_algorithms}

\subsection{Mirror-Dual Algorithms}

We first record the accelerated primal method used for the warm start.
The regularizer is centered at the supplied initial point. The proof of
\cref{lem:appendix_primal_warm_start} specifies the weights.

\begin{algorithm}[H]
\caption{Accelerated primal warm start}
\label{alg:primal-warm-start}
\footnotesize
\begin{algorithmic}[1]
\Require $f$, $\psi$, $x_0$, horizon $T$, weights $\{a_t\}_{t=1}^T$
\State $A_0\gets0$, $\widetilde z_0\gets\argmin_z\psi(z)$,
       $\widetilde y_0\gets\widetilde z_0$
\For{$t=1,\ldots,T$}
  \State $A_t\gets A_{t-1}+a_t$
  \State $\displaystyle
    \widetilde x_t\gets
    \frac{A_{t-1}}{A_t}\widetilde y_{t-1}
    +\frac{a_t}{A_t}\widetilde z_{t-1}$
  \State $\displaystyle
    \widetilde z_t\gets
    \argmin_z\left\{
      \sum_{i=1}^t a_i\innp{\nabla f(\widetilde x_i),z}+\psi(z)
    \right\}$
  \State $\displaystyle
    \widetilde y_t\gets
    \frac{A_{t-1}}{A_t}\widetilde y_{t-1}
    +\frac{a_t}{A_t}\widetilde z_t$
\EndFor
\State \Return $\widetilde y_T$
\end{algorithmic}
\end{algorithm}

We next make the coefficient arrays used by the mirror-dual construction
explicit. Define
\[
    v_t\defi-\sum_{i=1}^t a_i\nabla f(\widetilde x_i),
    \qquad v_0=0,
    \qquad z_t\defi\nabla\psi^*(v_t).
\]
Then the minimization defining $\widetilde z_t$ in
\cref{alg:primal-warm-start} gives $\widetilde z_t=z_t$.

\begin{lemma}[AGD iterate identities]
\label{lem:appendix-agd-identities}
For every $t\geq1$,
\[
    \widetilde y_t=\frac1{A_t}\sum_{i=1}^t a_i z_i.
\]
Moreover, $\widetilde x_1=z_0$, and for every $t\geq2$,
\[
    \widetilde x_t
    =\frac1{A_t}\sum_{i=1}^{t-2}a_i z_i
    +\frac{a_{t-1}+a_t}{A_t}z_{t-1}.
\]
\end{lemma}

\begin{proof}
The first identity follows by induction from the update for
$\widetilde y_t$. Substituting the identity at time $t-1$ into the
update for $\widetilde x_t$ gives the second identity.
\end{proof}

\begin{lemma}[Explicit CFOM coefficients]
\label{lem:appendix-agd-cfom-coefficients}
The iterates of \cref{alg:primal-warm-start} satisfy
\[
    v_t=v_{t-1}-a_t\nabla f(\widetilde x_t),
    \qquad t=1,\ldots,T,
\]
and
\[
    \widetilde x_{t+1}
    =\widetilde x_t-\sum_{i=0}^t b_{t+1,i}z_i,
    \qquad t=1,\ldots,T-1,
\]
where the first row is
\[
    b_{2,0}=1,
    \qquad b_{2,1}=-1,
\]
and, for $t\geq2$,
\[
    b_{t+1,i}
    =
    \begin{cases}
        \displaystyle
        \frac{a_i}{A_t}-\frac{a_i}{A_{t+1}},
        &1\leq i\leq t-2,\\[1.2ex]
        \displaystyle
        \frac{a_{t-1}+a_t}{A_t}
        -\frac{a_{t-1}}{A_{t+1}},
        &i=t-1,\\[1.2ex]
        \displaystyle
        -\frac{a_t+a_{t+1}}{A_{t+1}},
        &i=t,\\[1.2ex]
        0,&\text{otherwise}.
    \end{cases}
\]
Thus the gradient coefficient matrix has exactly one nonzero entry in
row $t$, equal to $a_t$, and the regularizer-oracle coefficient matrix
is given entrywise by the displayed $b_{t+1,i}$.
\end{lemma}

\begin{proof}
The recursion for $v_t$ follows directly from its definition. For the
first primal row, $\widetilde y_1=z_1$ and hence
$\widetilde x_2-\widetilde x_1=z_1-z_0$. For $t\geq2$,
\cref{lem:appendix-agd-identities} gives
\begin{align*}
\widetilde x_{t+1}
&=\frac1{A_{t+1}}\sum_{i=1}^{t-1}a_i z_i
  +\frac{a_t+a_{t+1}}{A_{t+1}}z_t,\\
\widetilde x_t
&=\frac1{A_t}\sum_{i=1}^{t-2}a_i z_i
  +\frac{a_{t-1}+a_t}{A_t}z_{t-1}.
\end{align*}
Subtracting the second identity from the first and reversing the signs
of the resulting coefficients gives the stated $b_{t+1,i}$.
\end{proof}

\begin{lemma}[Zero-sum property]
\label{lem:appendix-agd-zero-sum}
For every $t=1,\ldots,T-1$, the coefficients in
\cref{lem:appendix-agd-cfom-coefficients} satisfy
\[
    \sum_{i=0}^t b_{t+1,i}=0.
\]
\end{lemma}

\begin{proof}
For $t=1$, the claim follows from $b_{2,0}+b_{2,1}=1-1=0$.
For $t\geq2$, the explicit coefficients give
\begin{align*}
\sum_{i=0}^t b_{t+1,i}
&=
\frac{\sum_{i=1}^{t-2}a_i+a_{t-1}+a_t}{A_t}
-
\frac{\sum_{i=1}^{t-2}a_i+a_{t-1}+a_t+a_{t+1}}{A_{t+1}}\\
&=\frac{A_t}{A_t}-\frac{A_{t+1}}{A_{t+1}}=0.
\end{align*}
\end{proof}

To state the mirror-dual method compactly, write the coupled
first-order representation of \cref{alg:primal-warm-start} as
\begin{align}
v_{t+1}
&=v_t-\sum_{j=0}^t\alpha_{t+1,j}\nabla f(x_j),
\nonumber\\
x_{t+1}
&=x_t-\sum_{j=0}^{t+1}\beta_{t+1,j}\nabla\psi^*(v_j).
\label{eq:agd-cfom-appendix}
\end{align}
After the harmless shift from the one-based indexing of
\cref{lem:appendix-agd-cfom-coefficients} to the zero-based convention
in \eqref{eq:agd-cfom-appendix}, the only nonzero entries of $\alpha$
are the entries $a_t$ displayed above, and the entries of $\beta$ are
exactly the coefficients $b_{t+1,i}$. Hence the arrays depend only on
$\{a_t,A_t\}$ and are now fully specified.

\begin{proposition}[Gradient hitting]
\label{prop:appendix-gradient-hitting}
Let $(q_t,r_t)_{t=0}^T$ be the mirror-dual iterates obtained by
time-reversing the coefficient arrays in
\eqref{eq:agd-cfom-appendix}. Then
\[
    r_T=\nabla f(q_T),
    \qquad\text{and hence}\qquad
    \psi^*(r_T)=\psi^*(\nabla f(q_T)).
\]
\end{proposition}

\begin{proof}
By \cref{lem:appendix-agd-zero-sum}, every row of the
regularizer-oracle coefficient matrix has zero sum. The same is true
for the zero-based array $\beta$ in \eqref{eq:agd-cfom-appendix}.
Consequently, for every $k=0,\ldots,T-1$,
\[
    1-
    \sum_{i=0}^k\sum_{j=0}^{i+1}\beta_{i+1,j}
    =1.
\]
This is the sufficient coefficient condition in
\citet[Proposition~1]{kim2023mirror}.
Applying that proposition to the time-reversed CFOM yields
$r_T=\nabla f(q_T)$.
\end{proof}

The time reversal of the now explicit arrays gives the following
method.

\begin{algorithm}[H]
\caption{Mirror dual of accelerated gradient descent}
\label{alg:mirror-dual-agd}
\footnotesize
\begin{algorithmic}[1]
\Require $f$, $\psi$, $q_0$, horizon $T$, CFOM arrays
         $(\alpha,\beta)$ from \eqref{eq:agd-cfom-appendix}
\State $r_0\gets-\beta_{T,T}\nabla f(q_0)$
       \Comment{$r_0=0$ for \cref{alg:primal-warm-start}}
\For{$k=0,\ldots,T-1$}
  \State $\displaystyle
    q_{k+1}\gets q_k-
    \sum_{i=0}^k\alpha_{T-i,T-1-k}\nabla\psi^*(r_i)$
  \State $\displaystyle
    r_{k+1}\gets r_k-
    \sum_{i=0}^{k+1}\beta_{T-i,T-1-k}\nabla f(q_i)$
\EndFor
\State \Return $q_T$
\end{algorithmic}
\end{algorithm}

The parameters of the two calls differ: the primal weights minimize the
warm-start function gap, whereas the mirror-dual weights minimize the
dual potential in \cref{eq:appendix_dual_master}.

\begin{algorithm}[H]
\caption{Two-phase inexact mirror duality}
\label{alg:two-phase-mirror-dual}
\footnotesize
\begin{algorithmic}[1]
\Require $f$, $x_0$, $p$, $\kappa$, $L$,
         $R\geq\|x_0-x^*\|_p$, horizon $T$
\State $\psi_0(x)\gets\frac1p\|x-x_0\|_p^p$
\State Form the primal weights $\{a_t^{\mathrm P}\}$ using
       \eqref{eq:appendix_primal_gamma1} and
       \eqref{eq:appendix_primal_gamma2}, omitting inactive parameters
\State $x^{\mathrm w}\gets
       \Call{Primal-AGD}{f,\psi_0,x_0,T,\{a_t^{\mathrm P}\}}$
       \Comment{\cref{alg:primal-warm-start}}
\State $\psi(x)\gets\frac1p\|x\|_p^p$
\State Set $\overline\Delta$ to the warm-start bound from
       \cref{lem:appendix_primal_warm_start}
\State Form the dual weights $\{a_t^{\mathrm D}\}$ using
       \eqref{eq:appendix_dual_gamma2} and
       \eqref{eq:appendix_dual_gamma1}, with
       $\Delta\gets\overline\Delta$ and inactive parameters omitted.
       Form $(\alpha,\beta)$
\State $q\gets
       \Call{Mirror-Dual-AGD}{f,\psi,x^{\mathrm w},T,\alpha,\beta}$
       \Comment{\cref{alg:mirror-dual-agd}}
\State \Return $q$
\end{algorithmic}
\end{algorithm}

\section{Accumulative Power Regularization}
\label{app:accumulative-power}
We will also use the gradient form of H\"older smoothness. Let \(L_0\)
denote the constant in
\cref{def:uniform-convexity-holder-smoothness}. The generalized
cocoercivity inequality for convex H\"older-smooth functions
\citep{zalinescu2002convex} gives
\[
    D_f(x,y)
    \geq
    \frac{1}{\kappa^*}
    L_0^{-1/(\kappa-1)}
    \|\grad f(x)-\grad f(y)\|_*^{\kappa^*},
    \qquad
    \kappa^*=\frac{\kappa}{\kappa-1}.
\]
Combining this with
\(D_f(x,y)\leq (L_0/\kappa)\|x-y\|^\kappa\) yields
\[
    \|\grad f(x)-\grad f(y)\|_*
    \leq
    c_\kappa L_0\|x-y\|^{\kappa-1},
    \qquad
    c_\kappa=(\kappa-1)^{-(\kappa-1)/\kappa}.
\]
Since \(c_\kappa\geq1\) depends only on \(\kappa\), we set
\(L:=c_\kappa L_0\), which is still a valid H\"older-smoothness
constant and changes only constants depending on \(\kappa\).
Thus, with the convention used in
\cref{sec:accumulative_regularization},
\begin{equation}
\label{eq:gradient-holder}
    \|\grad f(x)-\grad f(y)\|_*
    \leq
    L\|x-y\|^{\kappa-1}.
\end{equation}

\begin{algorithm}[H]
    \caption{$\AGD+(x_0, \delta, R_0, f,\Psi,\phi)$ from \citep{diakonikolas2024complementary}}
\label{alg:generalized-agd}
\small
\algrenewcommand\algorithmicindent{1em}
\begin{algorithmic}[1]

\Require Initial point $x_0$, accuracy $\delta$, initial distance $R$ to a minimizer, $(L,\kappa)$-Hölder smooth function f, convex proximable function $\Psi$, regularizer $\phi$.
\State Set the following parameters as specified in \citep{diakonikolas2024complementary}: $\{(\alpha_j,c_j)\}_{j=1}^m$, $\bar y$, $\{a_k\}_{k=0}^N$, $m_0$ and number of iterations $N = C_{p,\kappa}\!\left( 1+ \left( \frac{L} {\lambda R^{q-\kappa}} \right)^{q/((q+1)\kappa-q)} \right)$, where $\lambda$ is the uniform-convexity constant of $H = f + \Psi$.

\vspace{0.1cm}
\hrule
\vspace{0.1cm}

\State $\Psi \gets \sum_{j=1}^m \alpha_j r_q(\cdot-c_j)$,
\quad
$\phi \gets \mu_{p,q}^{-1} r_q(\cdot-\bar y)$

\State $\Prox(g,A) \gets \arg\min_{y\in\R^d}
\bigl\{\langle g,y\rangle + A\Psi(y) + m_0\phi(y)\bigr\}$

\State $A_0 \gets a_0 \gets 1$, \quad $x_0 \gets \bar y$, \quad $g_0 \gets \nabla f(x_0)$
\State $v_0 \gets \Prox(g_0,A_0)$, \quad $y_0 \gets v_0$

\For{$k=1,\dots,N$}
    \State $A_k \gets A_{k-1}+a_k$
    \State $x_k \gets \bigl(A_{k-1}y_{k-1}+a_k v_{k-1}\bigr)/A_k$
    \State $g_k \gets g_{k-1}+a_k \nabla f(x_k)$
    \State $v_k \gets \Prox(g_k,A_k)$
    \State $y_k \gets \bigl(A_{k-1}y_{k-1}+a_k v_k\bigr)/A_k$
\EndFor

\State \Return $y_N$

\end{algorithmic}
\end{algorithm}

\begin{proof}[\linkofproof{thm:accumulative-power}]
First consider the endpoint cases. If \(p=1\), set
\(\widehat p=1+1/\log d\). By \cref{rem:endpoint-norms},
\(\|x\|_{\widehat p}\le\|x\|_1\le e\|x\|_{\widehat p}\), so
\((L,\kappa)\)-H\"older smoothness in \(\ell_1\) implies
\((e^\kappa L,\kappa)\)-H\"older smoothness in
\(\ell_{\widehat p}\), the \(\ell_{\widehat p}\)-radius is bounded by
the \(\ell_1\)-radius, and
\(\|g\|_\infty\le\|g\|_{\widehat p^\ast}\). Thus the finite
\(\widehat p\) result below gives the desired \(p=1\) guarantee after
absorbing universal constants.

    If \(p=\infty\), we can set \(\widehat p=\log d\) according to \cref{rem:endpoint-norms}. We may therefore assume
\(1<p<\infty\) for the rest of the proof.

We use the stage notation of \cref{alg:holder-ar}. Thus
\[
S\defi
\left\lceil
\frac{1}{\kappa-1}
\log_+\!\left(\frac{2LR_0^{\kappa-1}}{\eps}\right)
\right\rceil,\qquad
q\defi\max\{2,p\},\qquad
r_q(u)\defi\frac{1}{q}\pnorm{u}^q.
\]
We also write
\[
\tau\defi\frac{(q+1)\kappa-q}{q}.
\]
For \(1\le s\le S\), the algorithm sets
\[
R_s\defi R_0 2^{-s},\qquad
\sigma_0\defi0,\qquad
\sigma_1\defi
\frac{(1-2^{-(\kappa-1)/2})\eps}
{2\cdot3^{q-1}R_0^{q-1}},
\]
\[
\sigma_s\defi\sigma_1 2^{(s-1)(q-(\kappa+1)/2)},\qquad
\alpha_s\defi\sigma_s-\sigma_{s-1},\qquad
\Psi_s(x)\defi\sum_{i=1}^s\alpha_i r_q(x-x_{i-1}).
\]
We also set
\[
H_s(x)\defi f(x)+\Psi_s(x),\qquad
\phi_s(x)\defi\mu_{p,q}^{-1}r_q(x-x_{s-1}),\qquad
\delta_s\defi\frac{\mu_{p,q}\sigma_sR_s^q}{q},\qquad
\lambda_s\defi\mu_{p,q}\sigma_s,
\]
where \(\mu_{p,q}\) is the constant in \cref{fact:power_geometry}. The
points \(x_s\) are the outer iterates generated by \cref{alg:holder-ar},
and \(\lambda_s\) is the uniform-convexity constant of \(H_s\). Finally,
define
\[
    \beta\defi 2^{q-(\kappa+1)/2},
    \qquad
    \rho\defi\frac{\beta}{2^{q-1}}=2^{-(\kappa-1)/2}<1,
\]
so that \(\sigma_s=\sigma_1\beta^{s-1}\), and let
\(x_s^\ast=\argmin_x H_s(x)\) for \(1\le s\le S\).
Let \(N_s\) denote the value of \(N\) in the stage-\(s\) call to
\(\AGDp\) in \cref{alg:holder-ar}. By \cref{alg:generalized-agd}, that
call uses \(N_s+1\) gradient evaluations.

If \(2LR_0^{\kappa-1}\leq\eps\), then \(S=0\) in
\cref{alg:holder-ar}, so the algorithm returns \(x_0\). Moreover,
\eqref{eq:gradient-holder} and \(\grad f(x^\ast)=0\) give
\[
    \|\grad f(x_0)\|_{p^*}
    \leq
    L\|x_0-x^\ast\|_p^{\kappa-1}
    \leq LR_0^{\kappa-1}
    \leq \eps/2
    \leq \eps.
\]
Suppose henceforth that
$2LR_0^{\kappa-1}>\eps$.

We first bound the distance of the outer iterates with respect to the minimizers of the regularized problems.
Set $H_0=f$ and $x_0^\ast=x^\ast$. Recall
$H_s=H_{s-1}+(\alpha_s/q)\pnorm{\cdot-x_{s-1}}^q$. We prove by induction that
\begin{equation}\label{eq:iterate-localization}
  \pnorm{x_s-x_s^\ast}\le R_s,
  \qquad 0\le s\le S.
\end{equation}
The claim holds at $s=0$ because
$\pnorm{x_0-x^\ast}\le R_0$.  Suppose it holds at $s-1$. Then
\begin{align}\label{eq:center-localization}
 \begin{aligned}
  \frac{\alpha_s}{q}\pnorm{x_s^\ast-x_{s-1}}^q
  &=H_s(x_s^\ast)-H_{s-1}(x_s^\ast)
  \circled{1}[\le]
    H_s(x_{s-1}^\ast)-H_{s-1}(x_s^\ast)\\
  &=H_{s-1}(x_{s-1}^\ast)-H_{s-1}(x_s^\ast)
    +\frac{\alpha_s}{q}
      \pnorm{x_{s-1}^\ast-x_{s-1}}^q\\
  &\circled{2}[\le]
    \frac{\alpha_s}{q}\pnorm{x_{s-1}^\ast-x_{s-1}}^q \circled{3}[\le] \frac{\alpha_s}{q} R_{s-1}^q.
\end{aligned}   
\end{align}
Here $\circled{1}$ uses the optimality of $x_s^\ast$ for $H_s$,
    $\circled{2}$ uses the optimality of $x_{s-1}^\ast$ for $H_{s-1}$, and $\circled{3}$ uses the induction hypothesis. Moreover,
\[
  \frac{\lambda_s}{q}\pnorm{x_s-x_s^\ast}^q
  \circled{1}[\le]H_s(x_s)-H_s(x_s^\ast)
  \circled{2}[\le]\frac{\lambda_s}{q}R_s^q.
\]
Here $\circled{1}$ uses the $(\lambda_s,q)$-uniform convexity of the
accumulated regularizer from \cref{fact:power_geometry}, and
$\circled{2}$ applies \cref{fact:generalized-agd-inner} with
$\delta_s=\lambda_sR_s^q/q$ and the stage auxiliary regularizer $\phi_s(x)=(q\mu_{p,q})^{-1}\pnorm{x-x_{s-1}}^q$. This proves the induction claim.

Because \cref{eq:center-localization} holds at every stage $j$, for
$1\le i\le S$,
\begin{equation}\label{eq:accumulated-localization}
\begin{aligned}
  \pnorm{x_S^\ast-x_{i-1}}
  &\le \pnorm{x_i^\ast-x_{i-1}}
    +\sum_{j=i+1}^S\pnorm{x_j^\ast-x_{j-1}^\ast}\\
  &\le \pnorm{x_i^\ast-x_{i-1}}
    +\sum_{j=i+1}^S
      \left(
        \pnorm{x_j^\ast-x_{j-1}}
        +\pnorm{x_{j-1}-x_{j-1}^\ast}
      \right)\\
  &\circled{1}[\le]R_{i-1}+2\sum_{j=i+1}^S R_{j-1}
  \circled{2}[\le]3R_{i-1}.
\end{aligned}
\end{equation}
Here $\circled{1}$ uses
\cref{eq:iterate-localization,eq:center-localization}, while
$\circled{2}$ uses the geometric schedule $R_j=R_0 2^{-j}$.

We now show how we obtain the guarantee for the gradient norm.
The first-order optimality condition for $x_S^\ast$ is
\begin{equation}\label{eq:final-optimality}
  0=\grad f(x_S^\ast)
  +\sum_{i=1}^S\alpha_i\grad r_q(x_S^\ast-x_{i-1}).
\end{equation}
\[
\begin{aligned}
  \dnorm{\grad f(x_S)}
  &\le \dnorm{\grad f(x_S)-\grad f(x_S^\ast)}
    +\dnorm{\grad f(x_S^\ast)}\\
  &\circled{1}[\le]
  L\pnorm{x_S-x_S^\ast}^{\kappa-1}
  +\sum_{i=1}^S\alpha_i
    \dnorm{\grad r_q(x_S^\ast-x_{i-1})}\\
  &\circled{2}[\le]
  LR_S^{\kappa-1}
  +3^{q-1}
    \sum_{i=1}^S\alpha_iR_{i-1}^{q-1}\\
  &\circled{3}[\le]
  \frac{\eps}{2}
  +3^{q-1}\sigma_1R_0^{q-1}
    \sum_{i=1}^{\infty}
    \left(\frac{\beta}{2^{q-1}}\right)^{i-1}\\
  &\circled{4}[=]
  \frac{\eps}{2}
  +\frac{3^{q-1}\sigma_1R_0^{q-1}}{1-\rho}
  \circled{5}[=]\eps.
\end{aligned}
\]
Here \(\circled{1}\) uses \eqref{eq:gradient-holder} and
\cref{eq:final-optimality}.  For $\circled{2}$,
\cref{eq:iterate-localization,eq:accumulated-localization} are combined
with the identity
$\dnorm{\grad r_q(u)}=\pnorm{u}^{q-1}$ from
\cref{fact:power_geometry}. This explains the factor
$3^{q-1}R_{i-1}^{q-1}$.  Next, $\circled{3}$ uses
$LR_S^{\kappa-1}\le \eps/2$,
$\alpha_i\le\sigma_i=\sigma_1\beta^{i-1}$, and
$R_{i-1}=R_0 2^{-(i-1)}$.  Finally, $\circled{4}$ uses
$\rho=\beta/2^{q-1}<1$, and $\circled{5}$ is the definition of
$\sigma_1$.

Finally, we compute the gradient complexity of the algorithm. Let \(T\)
be the total number of gradient evaluations.
Since
\[
  \sigma_sR_s^{q-\kappa}
  =
  2^{-(q-\kappa)}\sigma_1R_0^{q-\kappa}
  \left(\frac{\beta}{2^{q-\kappa}}\right)^{s-1},
  \qquad
  \frac{\beta}{2^{q-\kappa}}=2^{(\kappa-1)/2}>1,
\]
\cref{fact:generalized-agd-inner} and the definition of $N_s$ give
\[
  N_s
  =
  O_{p,\kappa}\!\left(
    1+A\vartheta^{s-1}
  \right),
  \qquad
  \vartheta=2^{-\frac{\kappa-1}{2\tau}}<1,
\]
where
\[
  A
  =O_{p,\kappa}\!\left(
    \left(
      \frac{L}{\mu_{p,q}\sigma_1R_0^{q-\kappa}}
    \right)^{1/\tau}
  \right)
  =O_{p,\kappa}\!\left(
    \left(
      \frac{LR_0^{\kappa-1}}{\eps}
    \right)^{1/\tau}
  \right).
\]
Summing the geometric schedule gives
\[
\begin{aligned}
  T
  &=\sum_{s=1}^S(N_s+1)
  \circled{1}[\le]
  O_{p,\kappa}\!\left(
    S+\frac{A}{1-\vartheta}
  \right)
  \circled{2}[\le]
  O_{p,\kappa}\!\left(
    \left(\frac{LR_0^{\kappa-1}}{\eps}\right)^{1/\tau}
  \right)
  =
  O_{p,\kappa}\!\left(
    \left(\frac{LR_0^{\kappa-1}}{\eps}\right)^{
      \frac{q}{(q+1)\kappa-q}}
  \right).
\end{aligned}
\]
Here $\circled{1}$ sums the geometric bound on $N_s$. For
$\circled{2}$, let
\(
X:=\frac{2LR_0^{\kappa-1}}{\eps} > 1
\)
\(
\log_2 X
=\frac{\ln X}{\ln 2}
\le
\frac{\tau}{\ln 2}\bigl(X^{1/\tau}-1\bigr).
\)
Therefore
\(
S
=
\left\lceil
\frac{\log_2 X}{\kappa-1}
\right\rceil
=
O_{p,\kappa}(X^{1/\tau})
\).
Moreover, $(1-\vartheta)^{-1}=O_{p,\kappa}(1)$, so the bound on $A$
gives $\circled{2}$. The final equality uses
$1/\tau=q/((q+1)\kappa-q)$.
By \cref{alg:generalized-agd}, each inner iteration makes one gradient
evaluation and one shifted-power call, so the same bound holds for both
resources. The exponent is $2/(3\kappa-2)$ for $1<p\le2$ and
$p/((p+1)\kappa-p)$ for $p\ge2$. The endpoint exponents are obtained by
replacing \(p\) with the finite proxy \(\widehat p\) from
\cref{rem:endpoint-norms}.
\end{proof}

\section{Examples and experiments}\label{sec:regression_problems}

We now specify two problems that naturally are convex and smooth with respect to $p$-norms. For our experiments in \cref{fig:methods_comparison}, we sample data uniformly for the corresponding unit $p$-norm ball.

\subsection{Linear Regression in the \texorpdfstring{$p$}{p}-norm}
\[
C_p:=\max\left\{p-1,\frac{1}{p-1}\right\},
\qquad 
\|A\|_{p\to p}:=\sup_{\|x\|_p\le 1}\|Ax\|_p,
\qquad 
\frac1p+\frac{1}{p^\ast}=1 .
\]

\begin{proposition}[Smoothness of the squared residual in \(\ell_p\)]
\label{prop:linear-regression-smoothness}
Let \(1<p<\infty\), \(A\in \mathbb{R}^{m\times d}\), and \(b\in\mathbb{R}^m\). Define
\(
f(x):=\|Ax-b\|_p^2 .
\)
Then \(f\) is \(L\)-smooth with respect to \(\|\cdot\|_p\) with
\(
L \le 2C_p\|A\|_{p\to p}^2 .
\)
In particular, for \(p\ge 2\),
\(
L\le 2(p-1)\|A\|_{p\to p}^2 .
\)
\end{proposition}

\begin{proof}
We use the standard smoothness fact for \(\ell_p\): the map
\(
\Phi(z):=\frac12\|z\|_p^2
\)
satisfies
\begin{equation}\label{eq:1}
\|\nabla \Phi(u)-\nabla \Phi(v)\|_{p^\ast}
\le C_p\|u-v\|_p .
\end{equation}
Equivalently, \(z\mapsto \|z\|_p^2=2\Phi(z)\) is \(2C_p\)-smooth.

Now write
\[
f(x)=\|Ax-b\|_p^2 .
\]
By the chain rule,
\[
\nabla f(x)=A^\top \nabla \|\cdot\|_p^2(Ax-b).
\]
Therefore
\[
\begin{aligned}
    \|\nabla f(x)-\nabla f(y)\|_{p^\ast}
&=
\left\|A^\top\Big(\nabla\|\cdot\|_p^2(Ax-b)
-\nabla\|\cdot\|_p^2(Ay-b)\Big)\right\|_{p^\ast} \\
&\le
\|A^\top\|_{{p^\ast}\to {p^\ast}}
\left\|\nabla\|\cdot\|_p^2(Ax-b)
-\nabla\|\cdot\|_p^2(Ay-b)\right\|_{p^\ast} \\
&\le
2C_p\|A^\top\|_{{p^\ast}\to {p^\ast}}\|A(x-y)\|_p \\
&\le
2C_p\|A^\top\|_{{p^\ast}\to {p^\ast}}\|A\|_{p\to p}\|x-y\|_p .
\end{aligned}
\]
Finally, by duality,
\[
\|A^\top\|_{{p^\ast}\to {p^\ast}}=\|A\|_{p\to p}.
\]
Hence
\[
\|\nabla f(x)-\nabla f(y)\|_{p^\ast}
\le
2C_p\|A\|_{p\to p}^2\|x-y\|_p .
\]
\end{proof}

\subsection{Logistic regression}
\begin{proposition}[Smoothness of logistic regression in \(\ell_p\)]
\label{prop:logistic-regression-smoothness}
Let \(1<p<\infty\), let \(p^\ast\) be the dual exponent of \(p\), let
\(a_i\in\mathbb{R}^d\), \(y_i\in\{-1,1\}\), and define
\[
F(x):=\frac1n\sum_{i=1}^n
\log\bigl(1+\exp(-y_i\langle a_i,x\rangle)\bigr).
\]
If \(\|a_i\|_{p^\ast}\le 1\) for all \(i\), then \(F\) is \(1/4\)-smooth
with respect to \(\|\cdot\|_p\), i.e.
\[
\|\nabla F(x)-\nabla F(y)\|_{p^\ast}
\le
\frac14\|x-y\|_p .
\]
\end{proposition}

\begin{proof}
Let
\[
\ell_i(x):=\log\bigl(1+\exp(-y_i\langle a_i,x\rangle)\bigr),
\qquad
\phi(s):=\log(1+\exp(-s)).
\]
Then \(\ell_i(x)=\phi(y_i\langle a_i,x\rangle)\). A direct computation gives
\[
\phi''(s)=\frac{e^{-s}}{(1+e^{-s})^2}\le \frac14,
\qquad \forall s\in\mathbb{R}.
\]

By the chain rule,
\[
\nabla^2 \ell_i(x)
=
\phi''(y_i\langle a_i,x\rangle)\, a_i a_i^\top.
\]

Fix any direction \(h\in\mathbb{R}^d\). Using the above, we obtain
\[
\nabla^2 \ell_i(x)h
=
\phi''(y_i\langle a_i,x\rangle)\,\langle a_i,h\rangle\, a_i.
\]
Taking \(\|\cdot\|_{p^\ast}\),
\[
\|\nabla^2 \ell_i(x)h\|_{p^\ast}
=
\phi''(y_i\langle a_i,x\rangle)
|\langle a_i,h\rangle|\,\|a_i\|_{p^\ast}.
\]

Applying Hölder's inequality,
\[
|\langle a_i,h\rangle|
\le
\|a_i\|_{p^\ast}\|h\|_p,
\]
and using \(\phi''\le 1/4\) in the preceding Hessian expression,
\[
\|\nabla^2 \ell_i(x)h\|_{p^\ast}
\le
\frac14\,\|a_i\|_{p^\ast}^2\,\|h\|_p.
\]

Averaging over \(i\),
\[
\|\nabla^2 F(x)h\|_{p^\ast}
=
\left\|\frac1n\sum_{i=1}^n \nabla^2 \ell_i(x)h\right\|_{p^\ast}
\le
\frac1{4n}\sum_{i=1}^n \|a_i\|_{p^\ast}^2\,\|h\|_p.
\]

Under the assumption \(\|a_i\|_{p^\ast}\le 1\),
\[
\|\nabla^2 F(x)h\|_{p^\ast}
\le
\frac14 \|h\|_p.
\]

Finally, integrating the Hessian bound along the segment between \(x\) and \(y\),
\[
\|\nabla F(x)-\nabla F(y)\|_{p^\ast}
\le
\sup_{z}\sup_{\|h\|_p=1}
\|\nabla^2 F(z)h\|_{p^\ast}\,\|x-y\|_p
\le
\frac14\|x-y\|_p.
\]
\end{proof}

\end{document}